\documentclass[reqno,11pt]{amsart}
\pdfoutput1

\usepackage{amssymb,color,hyperref,mathrsfs,stmaryrd}
\usepackage{amsmath}

\usepackage[usenames,dvipsnames]{xcolor}

\usepackage[curve,matrix,arrow]{xy}

\numberwithin{equation}{section}

\newtheorem{lemma}[equation]{Lemma}

\newtheorem{prop}[equation]{Proposition}

\newtheorem{example}[equation]{Example}

\theoremstyle{definition}

\newtheorem{remark}[equation]{Remark}

\newtheorem{definition}[equation]{Definition}

\newcommand{\F}{\mathcal{F}}
\newcommand{\E}{\mathcal{E}}

\renewcommand{\L}{\mathcal{L}}

\newcommand{\M}{\mathcal{M}}
\newcommand{\N}{\mathcal{N}}

\newcommand{\D}{\mathbf{D}}
\renewcommand{\H}{\mathcal{H}}

\newcommand{\Hom}{\operatorname{Hom}}
\newcommand{\Aut}{\operatorname{Aut}}
\newcommand{\Inn}{\operatorname{Inn}}

\newcommand{\m}{\mathcal}
\newcommand{\ov}{\overline}

\newcommand{\id}{\operatorname{id}}

\newcommand{\One}{\operatorname{\mathbf{1}}}
\newcommand{\W}{\mathbf{W}}

\def \<{\langle }
\def \>{\rangle }

\renewcommand{\phi}{\varphi}
\title[Direct and central products of localities]{Direct and central products of localities}

\author[E.~Henke]{Ellen Henke}
\address{Institute of Mathematics, 
University of Aberdeen, Fraser Noble Building, Aberdeen AB24 3UE, U.K.}
\email{ellen.henke@abdn.ac.uk}

\begin{document}

\begin{abstract}
We develop a theory of direct and central products of partial groups and localities.
\end{abstract}

\maketitle


\smallskip


\bigskip

\section{Introduction}

Partial groups and localities were introduced by Chermak \cite{Chermak:2013}, in the context of his proof of the existence and uniqueness of centric linking systems. Roughly speaking, a partial group is a set $\L$ together with a product which is only defined on certain words in $\L$, and an inversion map $\L\rightarrow \L$ which is an involutory bijection, subject to certain axioms. A locality is a partial group equipped with some extra structure which makes it possible to define the fusion system of a locality. Essentially, localities are ``the same'' as the transporter systems of Oliver and Ventura \cite{Oliver/Ventura}; see the appendix to \cite{Chermak:2013}. As centric linking systems are special cases of transporter systems, the existence of centric linking systems implies that there is a locality attached to every fusion system. It is work in progress of Chermak and the author of this paper to build a local theory of localities similar to the local theory of fusion systems as developed by Aschbacher \cite{Aschbacher:2008}, \cite{Aschbacher:2011} based on earlier work of many other authors. The results we prove in this paper fit into this program. 

\smallskip

For fusion systems, a relatively canonical definition of an external direct product was already introduced by Broto, Levi and Oliver \cite{BLO2}. Building on this definition, Aschbacher \cite{Aschbacher:2011} introduced central products of fusion systems. In this paper, we develop a theory of direct and central products of partial groups and localities. Most of our definitions are again quite canonical. After some preliminaries, we introduce in  Section~\ref{SectionDirectProductExternalPartialGroups} direct products of partial groups and prove basic properties of these. This allows us in Section~\ref{SectionDirectProductExternalLocalities} to define external direct and central products of localities. In Section~\ref{SectionInternal} we introduce internal direct and central products of partial groups and localities, and we prove results relating them to their external counterparts. 

\smallskip

Of special interest are localities corresponding to centric linking systems or, more generally, linking localities as introduced in \cite{Henke:2015}. We prove that an external or internal direct product of two  localities is a linking locality if and only if the two localities we started with are linking localities. A similar result holds for central products. The reader is referred to  Lemma~\ref{DirectProductLinkingLocality}, Lemma~\ref{ExternalCentralProductLemma}(b) and Lemma~\ref{InternalCentralProductLinkingLocality} for the precise statements of the results.

\smallskip

Given a linking locality over a saturated fusion system $\F$, it is work in progress of Chermak and the author of this paper to prove that there is a one-to-one correspondence between the normal subsystems of $\F$ and the partial normal subgroups of the locality. A significant part of the theory developed in this paper is needed in the proof. In particular, at the end we prove   Proposition~\ref{LastProposition} with this application in mind.

\smallskip

Throughout, $p$ is always a prime.

\section{Background on fusion systems}

\subsection{Morphisms of fusion systems}

We refer the reader to \cite[Part~I]{Aschbacher/Kessar/Oliver:2011} for background on fusion systems. We moreover introduce the following terminology: Let $\F$ and $\F'$ be fusion systems over $S$ and $S'$ respectively. Then we say that a group homomorphism $\alpha\colon S\rightarrow S'$ induces a morphism from $\F$ to $\F'$ if, for each $\phi\in\Hom_\F(P,Q)$, there exists $\psi\in\Hom_{\F'}(P\alpha,Q\alpha)$ such that $(\alpha|_P)\psi=\phi(\alpha|_Q)$. Such $\psi$ is then uniquely determined, so $\alpha$ induces a map $\alpha_{P,Q}\colon\Hom_\F(P,Q)\rightarrow \Hom_{\F'}(P\alpha,Q\alpha)$. Together with the map $P\mapsto P\alpha$ from the set of objects of $\F$ to the set of objects of $\F'$ this gives a functor from $\F$ to $\F'$. Moreover, $\alpha$ together with the maps $\alpha_{P,Q}$ ($P,Q\leq S$) is a morphism of fusion systems in the sense of \cite[Definition~II.2.2]{Aschbacher/Kessar/Oliver:2011}. We call $(\alpha,\alpha_{P,Q}\colon P,Q\leq S)$ the morphism induced by $\alpha$. If $\E$ is a subsystem of $\F$ on $T\leq S$, then we denote by $\E\alpha$ the subsystem of $\F'$ on $T\alpha$ which is the image of $\E$ under the functor $\alpha^*$.

\smallskip

We say that $\alpha$ induces an epimorphism from $\F$ to $\F'$ if $(\alpha,\alpha_{P,Q}\colon P,Q\leq S)$ is a surjective morphism of fusion systems. This means that $\alpha$ is surjective as a map $S\rightarrow S'$ and, for all $P,Q\leq S$ with $\ker(\alpha)\leq P\cap Q$, the map $\alpha_{P,Q}$ is surjective, i.e. for each $\psi\in\Hom_{\F'}(P\alpha,Q\alpha)$, there exists $\phi\in\Hom_\F(P,Q)$ with $(\alpha|_P)\psi=\phi(\alpha|_Q)$. 
If $\alpha$ is in addition injective then we say that $\alpha$ induces an isomorphism from $\F$ to $\F'$. If $\alpha$ induces an isomorphism from $\F$ to $\F'$ then observe that the inverse map $\alpha^{-1}$ induces an isomorphism from $\F'$ to $\F$. Note also that the following property follows directly from the definitions:

\begin{remark}\label{EpiConjugates}
Let $\F$ and $\F'$ be fusion systems over $S$ and $S'$ respectively. Suppose that $\alpha$ induces an epimorphism from $\F$ to $\F'$. Let $\ker(\alpha)\leq P\leq S$ and $Q:=P\alpha$. Then $Q^\F=\{\hat{P}\alpha\colon \hat{P}\in P^\F\}$.
\end{remark}

The kernel of a group homomorphism $S\rightarrow S'$ which induces a morphism from the fusion system $\F$ to the fusion system $\F'$ is  always a strongly closed subgroup of $\F$. On the other hand, if $R$ is a strongly closed subgroup of $\F$, then there is a factor system $\F/R$ defined and the natural group homomorphism $S\rightarrow S/R$ is an epimorphism; see \cite[Section~II.5]{Aschbacher/Kessar/Oliver:2011} for details. If $\F$ is saturated and there exists an epimorphism from $\F$ to $\F'$, then $\F'$ is saturated. In particular, $\F/R$ is saturated for every strongly closed subgroup $R$ of $\F$.

\smallskip

If $\alpha$ induces an epimorphism from $\F$ to $\F'$ then one checks easily that the induced map
\[S/\ker(\alpha)\rightarrow S', \ker(\alpha)s\mapsto s\alpha\]
induces an isomorphism from $\F/\ker(\alpha)$ to $\F'$.

\begin{lemma}\label{CentralQuotient}
Let $\F$ and $\F'$ be saturated fusion systems over $S$ and $S'$ respectively, and suppose $\alpha\colon S\rightarrow S'$ induces an epimorphism from $\F$ to $\F'$ such that $\ker(\alpha)\leq Z(\F)$. Let $P\leq S$. Then the following hold:
 \begin{itemize}
 \item [(a)] We have $P\leq \F^{cr}$ if and only if $\ker(\alpha)\leq P$ and $P\alpha\in(\F')^{cr}$.
 \item [(b)] We have $P\in\F^s$ if and only if $P\alpha\in (\F')^s$.  
\end{itemize}  
\end{lemma}

\begin{proof}
By \cite[Lemma~3.6]{Henke:2015}, an isomorphism between two saturated fusion systems induces a bijection between the sets of subcentric subgroups of these two fusion systems. Similarly, such an isomorphism induces a bijection between the sets of centric radical subgroups of these two fusion systems. 

\smallskip

Set $Z:=\ker(\alpha)$. Then the map $\ov{\alpha}\colon \F/Z\rightarrow \F',\ker(\alpha)s\mapsto s\alpha$ induces an isomorphism between the two saturated fusion systems $\F/Z$ and $\F'$. Hence, $\ov{\alpha}$ induces a bijection between $(\F/Z)^s$ and $(\F')^s$, and between $(\F/Z)^{cr}$ and $(\F')^{cr}$. By \cite[Lemma~9.1]{Henke:2015}, $P\in\F^s$ if and only if $PZ/Z\in(\F/Z)^s$, and $P\in\F^{cr}$ if and only if $Z\leq P$ and $P/Z\in(\F/Z)^{cr}$. This implies the assertion.
\end{proof}

\subsection{External direct and central products of fusion systems}\label{DirectProductFusionSystemsSection}\hspace{5.0 cm}

\bigskip

\textbf{For the remainder of this section let $\F_i$ be a fusion system on $S_i$ for $i=1,2$.}

\bigskip

For each $i=1,2$ write $\pi_i\colon S_1\times S_2\rightarrow S_i,(s_1,s_2)\mapsto s_i$ for the projection map. 
Given $P_i,Q_i\leq S_i$ and $\phi_i\in\Hom_{\F_i}(P_i,Q_i)$ for each $i=1,2$, define an injective group homomorphism $\phi_1\times \phi_2\colon P_1\times P_2\rightarrow Q_1\times Q_2$ by
\[(x_1,x_2)(\phi_1\times \phi_2)=(x_1\phi_1,x_2\phi_2)\]
for all $x_1\in P_1$ and $x_2\in P_2$. The \textit{direct product} $\F_1\times \F_2$ is the fusion system on $S_1\times S_2$ which is generated by the maps of the form $\phi_1\times \phi_2$ with $P_i,Q_i\leq S_i$ and $\phi_i\in\Hom_{\F_i}(P_i,Q_i)$ for $i=1,2$. This means that every morphism in $\Hom_{\F_1\times \F_2}(P,Q)$ is of the form $(\phi_1\times\phi_2)|_P$ where $\phi_i\in \Hom_{\F_i}(P\pi_i,Q\pi_i)$ for $i=1,2$.

\smallskip
 
For $i=1,2$ let $\iota_i\colon S_i\rightarrow S_1\times S_2$ be the inclusion map, i.e. $s\iota_1=(s,1)$ and $s\iota_2=(1,s)$. Note that $\iota_i$ induces a morphism from $\F_i$ to $\F_1\times \F_2$. More precisely, the morphism induced by $\iota_1$ takes $\phi_1\in\Hom_{\F_1}(P,Q)$ to $\phi_1\times \id_{\{1\}}\in\Hom_{\F_1\times\F_2}(P\iota_1,Q\iota_1)$ for all $P,Q\leq S_1$, and the morphism induced by $\iota_2$ takes $\phi_2\in\Hom_{\F_2}(P,Q)$ to $\id_{\{1\}}\times \phi_2\in\Hom_{\F_1\times\F_2}(P\iota_2,Q\iota_2)$ for all $P,Q\leq S_2$. For $i=1,2$, we call the image $\F_i\iota_i$ the canonical image of $\F_i$ in $\F_1\times \F_2$ and denote it by $\hat{\F}_i$. Moreover, we set $\hat{S}_i=S_i\iota_i$. As $\iota_i$ is injective, $\hat{\F}_i\cong\F_i$ for $i=1,2$.

\begin{lemma}\label{DirectProductFusionSystems}
Let $\F=\F_1\times\F_2$ be the direct product of $\F_1$ and $\F_2$. Let $P_i\leq S_i$ for $i=1,2$.
\begin{itemize}
 \item [(a)] We have $P_1\times P_2\in\F^c$ if and only if $P_i\in\F_i^c$ for $i=1,2$.
 \item [(b)] $\Aut_\F(P_1\times P_2)\cong \Aut_{\F_1}(P_1)\times\Aut_{\F_2}(P_2)$ and $P_1\times P_2$ is radical in $\F$ if and only if $P_i$ is radical in $\F_i$ for $i=1,2$.
 \item [(c)] $\F^{cr}=\{R_1\times R_2\colon R_i\in\F_i^{cr}\}$.
 \item [(d)] We have $(P_1\times P_2)^\F=\{Q_1\times Q_2\colon Q_i\in P_i^{\F_i}\mbox{ for }i=1,2\}$. Moreover, $P_1\times P_2$ is fully $\F$-normalized if and only if $P_i$ is fully $\F_i$-normalized for each $i=1,2$.
 \item [(e)] If $P_i\in\F_i^s$ for $i=1,2$ then $P_1\times P_2\in\F^s$.
 \item [(f)] For $i=1,2$ let $\Delta_i$ be a set of subgroups of $S_i$ such that $\Delta_i$ is closed under taking $\F_i$-conjugates. Then $\Gamma:=\{R_1\times R_2\colon R_i\in\Delta_i\mbox{ for each }i=1,2\}$ is closed under taking $\F$-conjugates.
\end{itemize}
\end{lemma}

\begin{proof}
Property (a) follows from \cite[(2.6)(2),(3)]{Aschbacher:2011}. 

\smallskip

By the definition of $\F=\F_1\times \F_2$, the elements of $\Aut_\F(P_1\times P_2)$ are the automorphisms of the form $\phi_1\times\phi_2$ with $\phi_i\in\Aut_{\F_i}(P_i)$. This implies $\Aut_\F(P_1\times P_2)\cong \Aut_{\F_1}(P_1)\times\Aut_{\F_2}(P_2)$. For any two finite groups $G_1$ and $G_2$, $O_p(G_1\times G_2)=O_p(G_1)\times O_p(G_2)$. So $O_p(\Aut_\F(P_1\times P_2))\cong O_p(\Aut_{\F_1}(P_1))\times O_p(\Aut_{\F_2}(P_2))$. As $\Inn(P_1\times P_2)\cong \Inn(P_1)\times \Inn(P_2)$, it follows that $P_1\times P_2$ is radical in $\F$ if and only if $P_i$ is radical in $\F_i$ for $i=1,2$. This proves (b). 

\smallskip

By \cite[Lemma~3.1]{Andersen/Oliver/Ventura:2012}, every $\F$-centric $\F$-radical subgroup is of the form $R_1\times R_2$ with $R_i\leq S_i$ for $i=1,2$. Hence, property (c) follows from (a) and (b).

\smallskip

It follows from the definition of $\F_1\times \F_2$ that $(P_1\times P_2)^\F=\{Q_1\times Q_2\colon Q_i\in P_i^{\F_i}\mbox{ for }i=1,2\}$. Since $N_{S_1\times S_2}(P_1\times P_2)=N_{S_1}(P_1)\times N_{S_2}(P_2)$, this implies (d). Property (f) is a direct consequence of the first part of (d).

\smallskip

For the proof of (e) assume now that $P_i\in\F_i^s$ for $i=1,2$. Let $Q\in (P_1\times P_2)^\F$ be fully normalized. By (d), $Q=Q_1\times Q_2$ where $Q_i\in P_i^{\F_i}\cap \F_i^f$ for $i=1,2$. By \cite[(2.5)]{Aschbacher:2011}, $N_\F(Q)=N_{\F_1}(Q_1)\times N_{\F_2}(Q_2)$. Hence, by \cite[Proposition~3.4]{Andersen/Oliver/Ventura:2012}, we have $O_p(N_\F(Q))=O_p(N_{\F_1}(Q_1))\times O_p(N_{\F_2}(Q_2))$. For $i=1,2$, $O_p(N_{\F_i}(Q_i))$ is centric in $\F_i$ as $P_i\in\F_i^s$. Hence, by (a), $O_p(N_\F(Q))$ is centric in $\F$. Thus $P_1\times P_2$ is subcentric.   
\end{proof}

It is straightforward to observe that  
\[Z(\F_1\times \F_2)=Z(\F_1)\times Z(\F_2).\]
For any subgroup $Z\leq Z(\F_1)\times Z(\F_2)=Z(\F_1\times \F_2)$ such that $Z\cap \hat{S}_i=1$ for $i=1,2$, we call $(\F_1\times \F_2)/Z$ the \textit{(external) central product} of $\F_1$ and $\F_2$ (over $Z$). We write $\F_1\times_Z\F_2$ for this external central product. Set $S_1\times_Z S_2:=(S_1\times S_2)/Z$. If
  $\theta\colon S_1\times S_2\rightarrow S_1\times_ZS_2$ is the natural epimorphism, then $\theta|_{\hat{S}_i}$ is injective and $\theta|_{\hat{S}_i}$ is by \cite[(2.9)(3)]{Aschbacher:2011} an isomorphism from $\hat{\F}_i$ to $\ov{\F_i}:=\hat{\F}_i\theta$.

\subsection{Internal central products of fusion systems}

Suppose now that $\F$ is a fusion system over $S$ containing the fusion systems $\F_1$ and $\F_2$ as subsystems. So in particular, $S_i\leq S$ for $i=1,2$. 

\smallskip

We say that $\F$ is the \textit{(internal) central product} of $\F_1$ and $\F_2$ if $S_1\cap S_2\leq Z(\F_i)$ for $i=1,2$ and the map $\alpha\colon S_1\times S_2\rightarrow S,\;(x_1,x_2)\mapsto x_1x_2$ induces an epimorphism from $\F_1\times\F_2$ to $\F$ with $\hat{\F}_i\alpha=\F_i$ for $i=1,2$. 

\smallskip

Note here that $\alpha$ being a group homomorphism is equivalent to $[S_1,S_2]=1$ inside of $S$. Moreover, $\alpha$ being surjective is equivalent to $S=S_1S_2$. Suppose now that $\F$ is the internal central product of the subsystems $\F_1$ and $\F_2$. Set $Z:=\ker(\alpha)$. Then $\alpha$ induces an isomorphism of groups $\ov{\alpha}\colon S_1\times_ZS_2\rightarrow S$ via $xZ\mapsto x\alpha$. If $(x_1,x_2)\in Z$ then $x_1=x_2^{-1}\in S_1\cap S_2\leq Z(\F_i)$ for $i=1,2$. Hence, $Z\leq Z(\F_1)\times Z(\F_2)$. By definition of $\alpha$, $Z\cap \hat{S}_i=1$ for $i=1,2$. Therefore, $(\F_1\times\F_2)/Z$ is an external central product of $\F_1$ and $\F_2$. As $\alpha$ induces an epimorphism from $\F_1\times\F_2$ to $\F$ with $\hat{\F}_i\alpha=\F_i$ for $i=1,2$, $\ov{\alpha}$ induces an epimorphism from $\F_1\times_Z\F_2$ to $\F$ with $\ov{\F_i}\ov{\alpha}=\F_i$. As $\ov{\alpha}$ is a group isomorphism, $\ov{\alpha}$ is an isomorphism of fusion systems. So $\F$ is in a canonical way isomorphism to an external central product of $\F_1$ and $\F_2$.

\begin{lemma}\label{CentralProductFusionSystems}
 Let $\F$ be the internal central product of two subsystems $\F_1$ and $\F_2$. 
\begin{itemize}
 \item [(a)] $\F^{cr}=\{R_1R_2\colon R_i\in\F_i^{cr}\mbox{ for }i=1,2\}$.
 \item [(b)] Let $P_i\in\F_i^s$ for $i=1,2$. Then $P_1P_2\in\F^s$.
 \item [(c)] For $i=1,2$ let $\Delta_i$ be a set of subgroups of $S_i$ such that $\Delta_i$ is closed under taking $\F_i$-conjugates. Set $\Gamma:=\{P_1P_2\colon P_i\in\Delta_i\mbox{ for each }i=1,2\}$. Let $\Delta$ be the set of subgroups of $S$ containing an element of $\Gamma$. Then $\Gamma$ is closed under taking $\F$-conjugates, and $\Delta$ is closed under taking $\F$-conjugates and overgroups in $S$.
\end{itemize}
\end{lemma}

\begin{proof}
Property (a) follows from Lemma~\ref{CentralQuotient}(a) and Lemma~\ref{DirectProductFusionSystems}(c). Similarly, property (b) follows Lemma~\ref{CentralQuotient}(b) and Lemma~\ref{DirectProductFusionSystems}(e). Remark~\ref{EpiConjugates} and Lemma~\ref{DirectProductFusionSystems}(f) imply that $\Gamma$ is closed under taking $\F$-conjugates. Hence, $\Delta$ is closed under taking $\F$-conjugates as well. Clearly $\Delta$ is closed under taking overgroups in $S$.
\end{proof}

\section{Partial groups and localities}

\subsection{Partial groups}
Adapting the notation from \cite{Chermak:2013} and \cite{Chermak:2015}, we write $\W(\L)$ for the set of words in a set $\L$, $\emptyset$ for the empty word, and $v_1\circ v_2\circ\dots\circ v_n$ for the concatenation of words $v_1,\dots,v_n\in\W(\L)$. Moreover, we identify each element $f\in\L$ with the word $(f)\in\W(\L)$ of length one. Via this identification, we have in particular $\L\subseteq\W(\L)$. Roughly speaking, a partial group is a set $\L$ together with a product which is only defined on certain words in $\L$, and an inversion map $\L\rightarrow \L$ which is an involutory bijection, subject to certain axioms. We refer the reader to \cite[Definition~2.1]{Chermak:2013} or \cite[Definition~1.1]{Chermak:2015} for the precise definition of a partial group, and to the elementary properties of partial groups stated in \cite[Lemma~2.2]{Chermak:2013} or \cite[Lemma~1.4]{Chermak:2015}.

\bigskip

\textbf{For the remainder of this section let $\L$ be a partial group with product $\Pi\colon \D\rightarrow \L$ defined on the domain $\D\subseteq\W(\L)$.} 

\bigskip

It follows from the axioms of a partial group that $\emptyset\in\D$. We set $\One=\Pi(\emptyset)$. By \cite[Lemma~1.4(f)]{Chermak:2015}, we have $\One^{-1}=\One$. Given a word $v=(f_1,\dots,f_n)\in\D$, we write sometimes $f_1f_2\dots f_n$ for the product $\Pi(v)$.

\smallskip

If $\m{X}$ and $\m{Y}$ are subsets of $\L$, we set
\[\m{X}\m{Y}:=\{\Pi(x,y)\colon x\in X,\;y\in Y,\;(x,y)\in\D\}.\]

\smallskip

A partial subgroup of $\L$ is a subset $\H$ of $\L$ such that $f^{-1}\in\H$ for all $f\in\H$ and $\Pi(w)\in\H$ for all $w\in\W(\H)\cap\D$. Note that $\emptyset\in\W(\H)\cap\D$ and thus $\One=\Pi(\emptyset)\in\H$ if $\H$ is a partial subgroup of $\L$. It is easy to see that a partial subgroup of $\L$ is always a partial group itself whose product is the restriction of the product $\Pi$ to $\W(\H)\cap\D$. Observe furthermore that $\L$ forms a group in the usual sense if $\W(\L)=\D$; see \cite[Lemma~1.3]{Chermak:2015}. So it makes sense to call a partial subgroup $\H$ of $\L$ a \textit{subgroup of $\L$} if $\W(\H)\subseteq\D$. In particular, we can talk about \textit{$p$-subgroups of $\L$} meaning subgroups of $\L$ whose order is a power of $p$.

\begin{remark}\label{Ones}
Let $u,v\in\W(\L)$ such that $u\circ v\in\D$. Then $u\circ (\One)\circ v\in\D$ and $\Pi(u\circ (\One)\circ v)=\Pi(u\circ v)$.

\smallskip

As a consequence, if $w$ is a word whose entries are all $\One$, then $w\in\D$ and $\Pi(w)=\One$. So $\{\One\}$ is a subgroup of $\L$.
\end{remark}

\begin{proof}
The first part is shown in \cite[Lemma~1.4(c)]{Chermak:2015}. Using this property repeatedly starting with $u=v=\emptyset$, it follows that a word $w$ all of whose entries are $\One$ lies in $\D$ and that $\Pi(w)=\Pi(\emptyset)=\One$. As $\One^{-1}=\One$, it follows that $\{\One\}$ is a subgroup of $\L$.
\end{proof}

\subsection{Conjugation in partial groups}
For any $g\in\L$, $\D(g)$ denotes the set of $x\in\L$ with $(g^{-1},x,g)\in\D$. Thus, $\D(g)$ denotes the set of elements $x\in\L$ for which the conjugation $x^g:=\Pi(g^{-1},x,g)$ is defined. By the axioms of a partial group, $(g^{-1},g)\in\D$ and $\Pi(g^{-1},g)=\One$ for any $g\in\L$. So by Remark~\ref{Ones}, $(g^{-1},\One,g)\in\D$ and $\Pi(g^{-1},\One,g)=\One$. Hence, for any $g\in\L$, $\One\in\D(g)$ and $\One^g=\One$. As $\One^{-1}=\One$, it follows similarly by Remark~\ref{Ones} that $g\in\D(\One)$ and $g^{\One}=g$ for any $g\in\L$.

\smallskip

If $g\in\L$ and $X\subseteq \D(g)$ we set $X^g:=\{x^g\colon x\in X\}$. If we write $X^g$ for some $g\in\L$ and some subset $X\subseteq \L$, we will always implicitly mean that $X\subseteq\D(g)$. Similarly, if we write $x^g$ for $x,g\in\L$, we always mean that $x\in\D(g)$. 

\smallskip

If $X$ is a subsets of $\L$ then we set
\[N_\L(X):=\{g\in\L\colon X^g=X\}\mbox{ and }C_\L(X):=\{g\in\L\colon x^g=x\mbox{ for all }x\in X\}.\]
Note that $C_\L(X)\subseteq N_\L(X)$. Similarly, for $x\in \L$, we define $C_\L(x):=\{f\in\L\colon x^f=x\}$. As argued above, $\One$ is contained in the centralizer of any element or subset of $\L$.

\smallskip

If $X$ and $Y$ are subsets of $\L$ then set $N_Y(X)=N_\L(X)\cap Y$ and $C_Y(X)=C_\L(X)\cap Y$. Moreover, set
\[Z(\L):=C_\L(\L).\]

\begin{lemma}\label{Centralizers}
Let $f,g\in\L$. Then the following conditions are equivalent:
\begin{itemize}
 \item [(1)] $f\in C_\L(g)$.
 \item [(2)] $g\in C_\L(f)$.
 \item [(3)] $(f^{-1},g^{-1},f,g)\in\D$ and $f^{-1}g^{-1}fg=\One$.
 \item [(4)] $(g^{-1},f^{-1},g,f)\in\D$ and $g^{-1}f^{-1}gf=\One$.
\end{itemize}
Moreover, if $f\in C_\L(g)$ then $(f,g),(g,f)\in\D$ and $fg=gf$. 
\end{lemma}

\begin{proof}
We only need to prove that properties (1)-(4) are equivalent, since the assertion follows then from \cite[Lemma~1.5(b)]{Chermak:2015}.  Since the situation is symmetric in $f$ and $g$, it is sufficient to prove that (1) and (3) are equivalent, and that (3) implies (4). 

\smallskip

Assume first that (3) holds, i.e. that $u:=(f^{-1},g^{-1},f,g)\in\D$ and $\Pi(u)=\One$. Then by \cite[Lemma~1.4(f)]{Chermak:2015}, $(g^{-1},f^{-1},g,f)=u^{-1}\in\D$ and $\Pi(u^{-1})=\Pi(u)^{-1}=\One^{-1}=\One$. So (4) holds. By the axioms of a partial group, $(g^{-1},f,g)\in\D$ as $u\in\D$. By \cite[Lemma~1.4(d)]{Chermak:2015}, it follows moreover that $(f)\circ u\in\D$ and $\Pi(f\circ (u))=\Pi(g^{-1},f,g)=f^g$. Hence, $f^g=\Pi((f)\circ u)=\Pi(f,\Pi(u))=\Pi(f,\One)=f$ by the third axiom of a partial group and by Remark~\ref{Ones}. So (1) holds. This shows that (3) implies (1) and (4).

\smallskip

Assume now that (1) holds, i.e. $v=(g^{-1},f,g)\in\D$ and $f^g=\Pi(v)=f$. By the fourth axiom of a partial group,  $v^{-1}\circ v=(g^{-1},f^{-1},g,g^{-1},f,g)\in\D$ and $\Pi(v^{-1}\circ v)=\One$. Moreover, by \cite[Lemma~1.6(b)]{Chermak:2015}, $v^{-1}\in\D$ and $\Pi(v^{-1})=\Pi(v)^{-1}=f^{-1}$. Hence, by the axioms of a partial group, $(f^{-1},g^{-1},f,g)=(\Pi(v^{-1}))\circ v\in\D$ and $\Pi(f^{-1},g^{-1},f,g)=\Pi(v^{-1}\circ v)=\One$. So (1) implies (3). This completes the proof. 
\end{proof}

Since there is a natural notion of conjugation, there is also a natural notion of partial normal subgroups of partial groups. Namely, a partial subgroup $\N$ of $\L$ is called a \textit{partial normal subgroup} of $\L$ if $n^f\in\N$ for all $f\in\L$ and all $n\in\N\cap\D(f)$.

\subsection{Homomorphisms of partial groups}
In this subsection let $\L'$ be a partial group with domain $\D'$ and product $\Pi'\colon\D'\rightarrow\L'$. Let $\One'=\Pi'(\emptyset)$ be the identity in $\L'$.

\smallskip

If $\phi\colon M\rightarrow N$ is a map between two sets $M$ and $N$, then $\phi^*\colon\W(M)\rightarrow\W(N)$ denotes always the map induced by $\phi$, i.e. $(f_1,\dots,f_n)\phi^*=(f_1\phi,\dots,f_n\phi)$ for every word $(f_1,\dots,f_n)\in\W(M)$. 

\smallskip

Let $\beta\colon\L\rightarrow\L'$. Recall from \cite[Definition~1.11]{Chermak:2015} that $\beta$ is called a homomorphism of partial groups if $\D\beta^*\subseteq\D'$ and $\Pi'(v\beta^*)=(\Pi(v))\beta$ for all $v\in\D$. 

\smallskip

If $\beta$ is a homomorphism of partial groups, define the kernel of $\beta$ via
\[\ker(\beta)=\{f\in\L\colon f\beta=\One'\}.\]
By \cite[Lemma~1.14]{Chermak:2015}, the kernel of a homomorphism of partial groups forms always a partial normal subgroup.

\begin{definition}
 Let $\beta\colon \L\rightarrow \L'$ be a homomorphism of partial groups. We call $\beta$ a \textit{projection} of partial groups if $\D\beta^*=\D'$. A projection $\beta$ is called an \textit{isomorphism} of partial groups if $\beta$ is injective. 
\end{definition}

Note that the condition $\D\beta^*=\D'$ implies that $\beta$ is surjective, as every word of length one is an element of $\D'$. So every projection of partial groups is surjective as a map, and every isomorphism of partial groups is a bijection.

\begin{remark}\label{PartialSubgroupProjection}
 Let $\beta\colon\L\rightarrow \L'$ be a homomorphism of partial groups and let $\H$ be a partial subgroup of $\L$. Then $(\D\cap \W(\H))\beta^*\subseteq \D'\cap\W(\H\beta)$. Moreover, if $(\D\cap \W(\H))\beta^*=\D'\cap\W(\H\beta)$ then $\H\beta$ is a partial subgroup of $\L'$ and $\beta|_\H\colon \H\rightarrow\H\beta$ is a projection of partial groups.
\end{remark}

\begin{proof}
 Clearly, $(\D\cap \W(\H))\beta^*\subseteq \D'\cap\W(\H\beta)$. Assume now $(\D\cap \W(\H))\beta^*=\D'\cap\W(\H\beta)$. 
If $f\in\H\beta$ then $f=g\beta$ for some $g\in\H$. As $\H$ is a partial subgroup, $g^{-1}\in\H$. Thus, by \cite[Lemma~1.13]{Chermak:2015}, $f^{-1}=(g\beta)^{-1}=(g^{-1})\beta\in\H\beta$. Let now $v\in\D'\cap\W(\H\beta)=(\D\cap\W(\H))\beta^*$. Then there exists $u\in\D\cap\W(\H)$ such that $v=u\beta^*$. Then $\Pi'(v)=\Pi'(u\beta^*)=(\Pi(u))\beta$. As $\H$ is a partial subgroup of $\L$, we have $\Pi(u)\in\H$ and thus $\Pi(v)\in\H\beta$. Hence, $\H\beta$ is a partial subgroup of $\L'$. Clearly, $\beta|_\H$ is a projection of partial groups.
\end{proof}

\begin{remark}\label{IsomorphismOfPartialGroups}
Let $\beta\colon\L\rightarrow \L'$ be an isomorphism of partial groups. Then the following hold:
\begin{itemize}
 \item [(a)] The map $\beta^{-1}\colon\L'\rightarrow\L$ is an isomorphism of partial groups.
 \item [(b)] Let $\H\subseteq \L$. Then $\H$ is a partial subgroup of $\L$ if and only if $\H\beta$ is a partial subgroup of $\L'$.
\end{itemize}
\end{remark}

\begin{proof}
As $\beta$ is a bijection, $\beta^*$ is a bijection and $(\beta^{-1})^*=(\beta^*)^{-1}$. So $\D\beta^*=\D'$ implies $\D=\D'(\beta^{-1})^*$. In particular, for $v\in\D'$, we have $u:=v(\beta^{-1})^*\in\D$ and $u\beta^*=v$. So $\Pi'(v)=\Pi'(u\beta^*)=(\Pi(u))\beta$ implies $(\Pi'(v))\beta^{-1}=\Pi(u)=\Pi(v(\beta^{-1})^*)$. So (a) holds. 

\smallskip

For the proof of (b) let $\H$ be a partial subgroup of $\L$. As $\beta$ and $\beta^*$ are bijections, we have  $(\D\cap\W(\H))\beta^*=(\D\beta^*)\cap(\W(\H)\beta^*)=\D'\cap\W(\H\beta)$. So $\H\beta$ is a partial subgroup of $\L'$ by Remark~\ref{PartialSubgroupProjection}. Now (b) follows from (a).
\end{proof}

We call two partial groups isomorphic if there exists an isomorphism between them.

\subsection{Localities}

\begin{definition}
Let $\Delta$ be a set of subgroups of $\L$. We write $\D_\Delta$ for the set of words $(f_1,\dots,f_n)\in\W(\L)$ such that there exist $P_0,\dots,P_n\in\Delta$ with 
\begin{itemize}
\item [(*)] $P_{i-1}\subseteq \D(f_i)$ and $P_{i-1}^{f_i}=P_i$.
\end{itemize}
If $v=(f_1,\dots,f_n)\in\W(\L)$, then we say that $v\in\D_\Delta$ via $P_0,\dots,P_n$ (or $v\in\D$ via $P_0$), if $P_0,\dots,P_n\in\Delta$ and (*) holds.
\end{definition}

\begin{definition}\label{LocalityDefinition}
We say that $(\L,\Delta,S)$ is a \textit{locality} if the partial group $\L$ is finite as a set, $S$ is a $p$-subgroup of $\L$, $\Delta$ is a non-empty set of subgroups of $S$, and the following conditions hold:
\begin{itemize}
\item[(L1)] $S$ is maximal with respect to inclusion among the $p$-subgroups of $\L$.
\item[(L2)] $\D=\D_\Delta$.
\item[(L3)] For any subgroup $Q$ of $S$, for which there exist $P\in\Delta$ and $g\in\L$ with $P\subseteq \D(g)$ and $P^g\leq Q$, we have $Q\in\Delta$.
\end{itemize}
If $(\L,\Delta,S)$ is a locality and $v=(f_1,\dots,f_n)\in\W(\L)$, then we say that $v\in\D$ via $P_0,\dots,P_n$ (or $v\in\D$ via $P_0$), if $v\in\D_\Delta$ via $P_0,\dots,P_n$.
\end{definition}

If $\L$ is any partial group, $S$ a subset of $\L$, and $g\in\L$ we set 
\[S_g:=\{s\in S\cap\D(g)\colon s^g\in S\}.\]

\begin{lemma}[Important properties of localities]\label{LocalitiesProp}
Let $(\L,\Delta,S)$ be a locality. Then the following hold:
\begin{itemize}
\item [(a)] $N_\L(P)$ is a subgroup of $\L$ for each $P\in\Delta$.
\item [(b)] Let $P\in\Delta$ and $g\in\L$ with $P\subseteq S_g$. Then $P^g\in\Delta$. So in particular $Q$ is a subgroup of $S$. 
\end{itemize}
\end{lemma}

\begin{proof}
Property (a) is \cite[Lemma~2.3(a)]{Chermak:2015} and property (b) is \cite[Proposition~2.6(c)]{Chermak:2015}.
\end{proof}

Let $(\L,\Delta,S)$ be a locality. Then by \cite[Lemma~2.3(b)]{Chermak:2015}, for every $P\in\Delta$ and every $g\in\L$ with $P\subseteq S_g$, the map $c_g\colon P\rightarrow P^g,x\mapsto x^g$ is an injective group homomorphism. The fusion system $\F_S(\L)$ is the fusion system over $S$ generated by such conjugation maps. Equivalently, $\F_S(\L)$ is generated by the conjugation between subgroups of $S$.

\begin{definition}
 We say that the locality $(\L,\Delta,S)$ is a locality \textit{over $\F$} if $\F=\F_S(\L)$.
\end{definition}

\subsection{Projections of localities}\label{SubsectionLocalitiesProjections}

\begin{definition}
Let $\L$ and $\L'$ be partial groups, and let $\beta\colon \L\rightarrow\L'$ be a homomorphism of partial groups. For every set $\Gamma$ of subgroups of $\L$ we set
\[\Gamma\beta:=\{P\beta\colon P\in\Gamma\}.\]
Suppose now $(\L,\Delta,S)$ and $(\L',\Delta',S')$ form localities. Then $\beta$ is called a \textit{projection of localities} from $(\L,\Delta,S)$ to $(\L',\Delta',S')$ if $\beta$ is a projection of partial groups and $\Delta'=\Delta\beta$ (and thus also $S\beta=S'$). If $\beta$ is in addition injective then we call $\beta$ an \textit{isomorphism of localities} from $(\L,\Delta,S)$ to $(\L',\Delta',S')$.
\end{definition}

If $(\L,\Delta,S)$ is a locality, $\L'$ is a partial group and $\beta\colon \L\rightarrow\L'$ is a projection of partial groups then $(\L',\Delta\beta,S\beta)$ forms a locality by  \cite[Theorem~4.4]{Chermak:2015}. In other words, the projection $\beta$ of partial groups ``transports'' the locality structure on $\L$ to a locality structure on $\L'$. Clearly, $\beta$ is a projection of localities from $(\L,\Delta,S)$ to $(\L',\Delta\beta,S\beta)$. 

\smallskip

If $\beta$ is a bijection then actually the partial group structure on $\L$ can be ``transported'' as well. The following remark is straightforward to prove:

\begin{remark}\label{StructureTransport}
Suppose $\L$ is a partial group as before, $\L'$ is a set and $\beta\colon\L\rightarrow\L'$ is a bijection. Notice that then $\beta^*$ is a bijection as well. We can turn $\L'$ into a partial group by setting $\D':=\{v\beta^*\colon v\in\D\}$, $\Pi'(v\beta^*):=(\Pi(v))\beta$ for every $v\in\D$ and $(f\beta)^{-1}=(f^{-1})\beta$ for every $f\in\L$. By construction, $\beta$ is then an isomorphism of partial groups from $\L$ to the newly constructed partial group $\L'$. 

\smallskip

If $(\L,\Delta,S)$ is a locality then $(\L',\Delta\beta,S\beta)$ is a locality. Moreover, $\beta$ is an isomorphism of localities from $(\L,\Delta,S)$ to $(\L',\Delta\beta,S\beta)$. 
\end{remark}

\smallskip

Chermak \cite{Chermak:2015} developed a theory of quotient localities modulo partial normal subgroups. We refer the reader to this article for details, but give a quick summary here: Suppose $(\L,\Delta,S)$ is a locality and $\N$ is a partial normal subgroup of $\L$. For $f\in\L$ set
\[\N f:=\{\Pi(n,f)\colon n\in\N,\;(n,f)\in\D\}\]
and call $\N f$ a right coset of $\N$ in $\L$. If $\N f$ is maximal with respect to inclusion among the right cosets of $\N$ in $\L$, then we call $\N f$ a \textit{maximal (right) coset}. The maximal right cosets form by \cite[Proposition~3.14(d)]{Chermak:2015} a partition of $\L$, i.e. every element of $\L$ lies in a unique maximal right coset. The map
\[\L\rightarrow \L/\N\]
mapping every element $g\in\L$ to the unique maximal right coset of $\N$ containing $g$ is a projection of partial groups; see \cite[Corollary~4.6]{Chermak:2015}. It is called the \textit{canonical projection} $\L\rightarrow\L/\N$.
 The kernel of the canonical projection equals $\N$. If $\beta$ is as above and $\N=\ker(\beta)$ then the map
\[\L/\N\rightarrow \L',\;\N f\mapsto f\beta\]
is by \cite[Theorem~4.7]{Chermak:2015} well-defined and an isomorphism of partial groups.

\smallskip

Chermak \cite[Definition~3.6]{Chermak:2015} defines $\uparrow$-maximal elements of $\L$ (relative to $\N$). We will not work directly with the definition of $\uparrow$-maximal elements here, but only use the following characterization: For any $f\in\L$, the right coset $\N f$ is a maximal coset if and only if $f$ is $\uparrow$-maximal relative to $\N$ (cf \cite[Proposition~3.14(c)]{Chermak:2015}).

\begin{lemma}\label{ModCentral1}
Let $(\L,\Delta,S)$ be a locality, $\L'$ a partial group, and suppose $\beta\colon \L\rightarrow\L'$ is a projection of partial groups. Assume $\N:=\ker(\beta)\subseteq Z(\L)$. Then every coset of $\N$ in $\L$ has $|\N|$ elements and is thus maximal. Moreover, for all $v\in\W(\L)$, we have $v\in\D$ if and only if $v\beta^*\in\D'$.
\end{lemma}

\begin{proof}
Let $f\in\L$. Since $\N\subseteq Z(\L)$, we have $(n,f)\in\D$ for all $n\in\N$ by Lemma~\ref{Centralizers}. So we have a well-defined map
\[\N\rightarrow \N f,\;n\mapsto \Pi(n,f)\]                                                            
and this map is clearly surjective. If $\Pi(n,f)=\Pi(n',f)$ with $n,n'\in\N$, then the right cancellation rule \cite[Lemma~1.4(e)]{Chermak:2015} yields $n=n'$. Hence, the above map is a bijection showing that every coset has precisely $|\N|$ elements. Hence, every right coset is maximal with respect to inclusion among the right cosets of $\N$. So every element of $\L$ is $\uparrow$-maximal. If $v=(f_1,\dots,f_n)\in\W(\L)$ such that every $f_i$ is $\uparrow$-maximal, then by \cite[Theorem~4.4(b)]{Chermak:2015}, $v\in\D$ if and only if $v\beta^*\in\D'$. This implies the assertion.                         
\end{proof}

\begin{lemma}\label{LocalitiesProjectionsModCentral}
 Let $(\L,\Delta,S)$ be a locality, let $\L'$ be a partial group, and let $\beta\colon \L\rightarrow\L'$ be a projection of partial groups with $\ker(\beta)\subseteq Z(\L)$. Suppose $\H$ is a partial subgroup of $\L$. Then  $\H\beta$ is a partial subgroup of $\L'$. Moreover, $(\D\cap\W(\H))\beta^*=\D'\cap \W(\H\beta)$, i.e. the restriction of $\beta$ to a map $\H\rightarrow \H\beta$ is a projection of partial groups. 
\end{lemma}

\begin{proof}
By Remark~\ref{PartialSubgroupProjection}, it is sufficient to show that $\D'\cap\W(\H\beta)\subseteq (\D\cap\W(\H))\beta^*$. Let $w=(f_1,\dots,f_m)\in\D'\cap\W(\H\beta)$. Then, for every $i=1,\dots,m$, there exists $h_i\in\H$ such that $h_i\beta=f_i$. So for $v:=(h_1,\dots,h_m)$ we have $v\in\W(\H)$ and $v\beta^*=w\in\D'$. Hence, by Lemma~\ref{ModCentral1}, $v\in\D$ and $w=v\beta^*\in(\D\cap\W(\H))\beta^*$.

\end{proof}

\begin{lemma}\label{LocalitiesProjectionsPartialNormal}
Let $(\L,\Delta,S)$ be a locality, let $\L'$ be a partial group, and let $\beta\colon \L\rightarrow\L'$ be a projection of partial groups. Let $\N$ be a partial normal subgroup of $\L$. Then $\N\beta$ is a partial normal subgroup of $\L'$. 
\end{lemma}

\begin{proof}
By \cite[Lemma~1.14]{Chermak:2015}, $\ker(\beta)$ forms a partial normal subgroup of $\L$. So by \cite[Theorem~1]{Henke:2015}, $\M:=\N(\ker\beta)$ is a partial normal subgroup of $\L$. Note that $\N\subseteq \M$ as, for any $n\in\N$, $n=\Pi(n)=\Pi(n,\One)\in\M$ by Remark~\ref{Ones}. Similarly one shows $\ker(\beta)\subseteq \M$. As $\N\subseteq\M$ we have $\N\beta\subseteq\M\beta$. Let $n\in\N$ and $x\in\ker(\beta)$ such that $(n,x)\in\D$. Then $\Pi(n,x)\beta=\Pi'(n\beta,x\beta)=\Pi'(n\beta,\One')=\Pi'(n\beta)=n\beta$, where the first equality uses that $\beta$ is a homomorphism of partial groups and the third equality uses Remark~\ref{Ones}. Hence, $\Pi(n,x)=n\beta\in \N\beta$. This shows $\M\beta=\N\beta$. As $\ker(\beta)\subseteq \M$, it follows from \cite[Proposition~4.8]{Chermak:2015} that $\N\beta=\M\beta$ is a partial normal subgroup of $\L'$. 
\end{proof}

\subsection{Sublocalities}

\begin{definition}
Let $\L_0$ be a partial subgroup of $\L$. We say that $(\L_0,\Delta_0,S_0)$ is a \textit{sublocality} of $(\L,\Delta,S)$ if $\L_0$ is a partial subgroup of $\L$, $S_0=S\cap \L_0$, $\Delta_0$ is a set of subgroups of $S_0$ and, regarding $\L_0$ as a partial group with product $\Pi|_{\W(\L_0)\cap\D}$, the triple $(\L_0,\Delta_0,S_0)$ forms a locality. 
\end{definition}

Supposing $\L_0$ is a partial subgroup of $\L$, we remark that $S_0:=S\cap \L_0$ is always a subgroup of $S$ and thus a $p$-subgroup of $\L_0$. If $\Delta_0$ is a non-empty set of subgroups of $S_0$ the $(\L_0,\Delta_0,S_0)$ forms a locality if and only if $\D\cap \W(\L_0)=\D_{\Delta_0}\cap\W(\L_0)$, $\Delta_0$ is closed under taking $\L_0$-conjugates and overgroups in $S_0$, and $S_0$ is maximal with respect to inclusion among the $p$-subgroups of $\L_0$.

\begin{lemma}\label{SublocalityUnderPartialHom}
 Let $(\L',\Delta',S')$ be a locality, and let $\beta\colon \L\rightarrow \L'$ be a homomorphism of partial groups. Let $(\L_0,\Delta_0,S_0)$ be a sublocality of $(\L,\Delta,S)$ with $S_0\beta\subseteq S'$ and $(\D\cap\W(\L_0))\beta^*=\D'\cap \W(\L_0\beta)$. Then $(\L_0\beta,\Delta_0\beta,S_0\beta)$ is a sublocality of $(\L',\Delta',S')$, and $\beta|_{\L_0}\colon \L_0\rightarrow\L_0\beta$ is a projection of localities from $(\L_0,\Delta_0,S_0)$ to $(\L_0\beta,\Delta_0\beta,S_0\beta)$.
\end{lemma}

\begin{proof}
By Remark~\ref{PartialSubgroupProjection}, $\L_0\beta$ is a partial subgroup of $\L'$ and $\beta|_{\L_0}\colon \L_0\rightarrow \L_0\beta$ is a projection of partial groups. 
So by \cite[Theorem~4.4]{Chermak:2015}, $(\L_0\beta,\Delta_0\beta,S_0\beta)$ is a locality. In particular, $S_0\beta$ is a maximal $p$-subgroup of $\L_0\beta$. As $\L_0\beta$ is a partial subgroup of $\L'$, $S'\cap (\L_0\beta)$ is a subgroup of $S'$ and thus a $p$-subgroup of $\L_0\beta$. Since $S_0\beta\subseteq S'\cap (\L_0\beta)$ it follows $S_0\beta=S'\cap(\L_0\beta)$. So $(\L_0\beta,\Delta_0\beta,S_0\beta)$ is a sublocality of $(\L',\Delta',S')$. Clearly, $\beta|_{\L_0}$ is a projection of localities from $(\L_0,\Delta_0,S_0)$ to $(\L_0\beta,\Delta_0\beta,S_0\beta)$.
\end{proof}

\begin{lemma}\label{SublocalityUnderProjection}
Let $(\L',\Delta',S')$ be a locality and let $\beta\colon\L\rightarrow\L'$ be a projection from $(\L,\Delta,S)$ to  $(\L',\Delta',S')$. Assume $\ker(\beta)\subseteq Z(\L)$. Let $(\L_0,\Delta_0,S_0)$ be a sublocality of $(\L,\Delta,S)$. Then $(\L_0\beta,\Delta_0\beta,S_0\beta)$ is a sublocality of $(\L',\Delta',S')$ and $\beta|_{\L_0}\colon \L_0\rightarrow \L_0\beta$ is a projection of localities from $(\L_0,\Delta_0,S_0)$ to $(\L_0\beta,\Delta_0\beta,S_0\beta)$. 
\end{lemma}

\begin{proof}
 By Lemma~\ref{LocalitiesProjectionsModCentral}, we have $(\D\cap\W(\L_0))\beta^*=\D'\cap \W(\L_0\beta)$. As $\beta$ is a projection of localities, $S_0\beta\subseteq S\beta=S'$. Hence, the assertion follows from Lemma~\ref{SublocalityUnderPartialHom}. 
\end{proof}

\subsection{Linking localities}

For the convenience of the reader we repeat the following definitions from \cite{Henke:2015}.

\begin{definition}
\noindent\begin{itemize}
\item A finite group $G$ is said to be of \textit{characteristic $p$} if $C_G(O_p(G))\leq O_p(G)$.
\item Define a locality $(\L,\Delta,S)$ to be of \textit{objective characteristic $p$} if, for any $P\in\Delta$, the group $N_\L(P)$ is of characteristic $p$. 
\item A locality $(\L,\Delta,S)$ is called a \textit{linking locality}, if $\F_S(\L)^{cr}\subseteq \Delta$ and $(\L,\Delta,S)$ is of objective characteristic $p$.
\end{itemize}
\end{definition}
 
If $(\L,\Delta,S)$ is a linking locality over $\F$, then it turns out that $\Delta\subseteq\F^s$ where $\F^s$ is defined as follows:

\begin{definition}
Let $\F$ be a fusion system over $S$. A subgroup $Q\leq S$ is said to be \textit{subcentric} in $\F$ if, for any fully normalized $\F$-conjugate $P$ of $Q$, $O_p(N_\F(P))$ is centric in $\F$. Write $\F^s$ for the set of subcentric subgroups of $\F$.
\end{definition}

Let $\F$ be a saturated fusion system over $S$. In \cite[Theorem~A]{Henke:2015} we show that the set $\F^s$ is closed under taking $\F$-conjugates and overgroups in $S$. Moreover, if $\F^{cr}\subseteq\Delta\subseteq\F^s$ and $\Delta$ is closed under taking $\F$-conjugates and overgroups in $S$, then we prove that there exists a linking locality over $\F$ with object set $\Delta$ which is essentially unique.

\begin{prop}\label{IsoFusionIsoLinkingLocality}
 Let $\F$ and $\F'$ be fusion systems over $S$ and $S'$ respectively. Let $(\L,\Delta,S)$ be a linking locality over $\F$, and let $(\L',\Delta',S')$ be a linking locality over $\F'$. Suppose $\alpha\colon S\rightarrow S'$ induces an isomorphism $\F$ rightarrow $\F'$. Assume furthermore that $\Delta\alpha=\Delta'$. Then there exists $\beta\colon \L\rightarrow\L'$ such that $\beta$ is an isomorphism of localities from $(\L,\Delta,S)$ to $(\L',\Delta',S')$ with $\beta|_S=\alpha$. 
\end{prop}

\begin{proof}
 By Remark~\ref{StructureTransport}, we can replace the set $\L$ by another isomorphic set if necessary and assume without loss of generality that $(\L\backslash S)\cap S'=\emptyset$. Set $\hat{\L}:=(\L\backslash S)\cup S'$. Then the map $\rho\colon \L\rightarrow\hat{\L}$ with $\rho|_{\L\backslash S}=\id$ and $\rho|_S=\alpha$ is a bijection. Hence, by Remark~\ref{StructureTransport}, we can turn $\hat{\L}$ into a partial group such that $(\hat{\L},\Delta',S')$ is a locality and $\rho$ is an isomorphism from the locality $(\L,\Delta,S)$ to the locality $(\hat{\L},\Delta',S')$. Then by \cite[Theorem~5.7(b)]{Henke:2015}, $\rho|_S=\alpha\colon S\rightarrow S'$ induces an isomorphism from $\F=\F_S(\L)$ to $\F_{S'}(\hat{\L})$. As $\alpha$ induces an isomorphism from $\F$ to $\F'$, it follows $\F_{S'}(\hat{\L})=\F'$. Since $(\L,\Delta,S)$ is a linking locality, $(\hat{\L},\Delta',S')$ is a linking locality as well. So $(\hat{\L},\Delta',S')$ and $(\L',\Delta',S')$ are both linking localities over $\F'$. Hence, by \cite[Theorem~A(a)]{Henke:2015}, there exists a rigid isomorphism $\gamma\colon \hat{\L}\rightarrow \L'$ (i.e. an isomorphism with $\gamma|_{S'}=\id_{S'}$). Then $\beta:=\rho\circ \gamma\colon \L\rightarrow\L'$ is an isomorphism of localities from $(\L,\Delta,S)$ to $(\L',\Delta',S')$ with $\beta|_S=(\rho|_S)\circ(\gamma|_{S'})=\alpha\circ\id_{S'}=\alpha$. 
\end{proof}

\section{External direct products of partial groups} \label{SectionDirectProductExternalPartialGroups}

For $i=1,2$ let $\L_i$ be a partial group with product $\Pi_i\colon \D_i\rightarrow \L_i$ and inversion map $\L_i\rightarrow \L_i,f\mapsto f^{-1}$. Let
\[\L=\L_1\times\L_2=\{(f,g)\colon f\in\L_1,\;g\in\L_2\}.\]
be the set theoretic product of $\L_1$ with $\L_2$. We will define a partial product and an inversion map on $\L$ which turns $\L$ into a partial group.

\smallskip

For any word $u=((f_1,g_1),(f_2,g_2),\dots,(f_n,g_n))\in\W(\L)$, we set $u_1:=(f_1,f_2,\dots,f_n)$ and $u_2:=(g_1,g_2,\dots,g_n)$. If $u=\emptyset$ then we mean here $u_i=\emptyset$ for $i=1,2$. So in any case, $u_i\in\W(\L_i)$ for $i=1,2$. 
Set
\[\D=\{u\in\W(\L)\colon u_i\in\D_i\mbox{ for each }i=1,2\}.\]
Define
\[\Pi\colon \D\rightarrow \L,\;u\mapsto (\Pi_1(u_1),\Pi_2(u_2)).\]
Note that in particular, $\One:=\Pi(\emptyset)=(\Pi_1(\emptyset),\Pi_2(\emptyset))=(\One,\One)$ (where $\One$ denotes also $\Pi_i(\emptyset)$ for $i=1,2$). 
If $f=(f_1,f_2)\in\L$ with $f_i\in\L_i$ for $i=1,2$, set 
\[f^{-1}=(f_1^{-1},f_2^{-1}).\]

\begin{lemma}
 The set $\L=\L_1\times\L_2$ with the partial product $\Pi\colon\D\rightarrow \L$ and the inversion defined above forms a partial group. 
\end{lemma}

\begin{proof}
If $u,v\in\W(\L)$ then note that $(u\circ v)_i=u_i\circ v_i$ for $i=1,2$ and similarly, $(u\circ v\circ w)_i=u_i\circ v_i\circ w_i$ for $u,v,w\in\W(\L)$. We will use this property throughout.

\smallskip

 As $\L_i$ is a partial group, $\L_i\subseteq\D_i$ and $\Pi_i|_{\L_i}=\id_{\L_i}$ for $i=1,2$. Since $(f)_i=(f_i)$ for any $f=(f_1,f_2)\in\L$, it follows from the definition of $\D$ that $\L\subseteq\D$. Moreover, for any $f=(f_1,f_2)\in\L$, we have $\Pi(f)=(\Pi_1(f_1),\Pi_2(f_2))=(f_1,f_2)$ by definition of $\Pi$. Hence, $\Pi|_\L=\id_\L$. If $u,v\in\W(\L)$ such that $u\circ v\in\D$, then it follows from the definition of $\D$ that $u_i\circ v_i=(u\circ v)_i\in\D_i$ for $i=1,2$. Thus, as $\L_i$ is a partial group with domain $\D_i$, we have $u_i,v_i\in\D_i$ for $i=1,2$. Hence, again by the definition of $\D$, it follows $u,v\in\D$. This proves the first and the second axiom of a partial group.

\smallskip

Note now that $\Pi(v)_i=\Pi_i(v_i)$ (or more precisely $(\Pi(v))_i=(\Pi_i(v_i))$) for any $v\in\D$ and any $i=1,2$ as $\Pi(v)=(\Pi_1(v_1),\Pi_2(v_2))$. Let $u,v,w\in\W(\L)$ such that $u\circ v\circ w\in\D$. Then for $i=1,2$, we have $u_i\circ v_i\circ w_i=(u\circ v\circ w)_i\in\D_i$ by definition of $\D$. Hence, as $\L_i$ is a partial group, $u_i\circ (\Pi_i(v_i))\circ w_i\in\D_i$ and $\Pi_i(u_i\circ v_i\circ w_i)=\Pi_i(u_i\circ (\Pi_i(v_i))\circ w_i)$. Thus $(u\circ (\Pi(v))\circ w)_i=u_i\circ (\Pi(v))_i\circ w_i=u_i\circ (\Pi_i(v_i))\circ w_i\in\D_i$ for $i=1,2$. Again by the definition of $\D$, it follows $u\circ (\Pi(v))\circ w\in\D$. Moreover, $\Pi(u\circ (\Pi(v))\circ w)_i=\Pi_i((u\circ (\Pi(v))\circ w)_i)=\Pi_i(u_i\circ (\Pi_i(v_i))\circ w_i)=\Pi_i(u_i\circ v_i\circ w_i)=\Pi_i((u\circ v\circ w)_i)=\Pi(u\circ v\circ w)_i$ for $i=1,2$. Hence, $\Pi(u\circ \Pi(v)\circ w)=\Pi(u\circ v\circ w)$. This proves the third axiom of a partial group for $\L$.

\smallskip

As the inversion maps on $\L_i$ is an involutory bijection for each $i=1,2$, the inversion map $\L\rightarrow \L,f\mapsto f^{-1}$ is also an involutory bijection. Note that for any $w\in\W(\L)$, $(w^{-1})_i=(w_i)^{-1}$. If $w\in\D$, then by definition of $\D$, $w_i\in\D_i$ for $i=1,2$. Thus, by the axioms of a partial group for $\L_i$, $(w_i)^{-1}\circ w_i\in\D_i$ and $\Pi_i((w_i)^{-1}\circ w_i)=\One$. Hence, $(w^{-1}\circ w)_i=(w^{-1})_i\circ w_i=(w_i)^{-1}\circ w_i\in\D_i$ for each $i=1,2$. So again by definition of $\D$, $w^{-1}\circ w\in\D$. Moreover, $\Pi(w^{-1}\circ w)=(\Pi_1((w_1)^{-1}\circ w_1),\Pi_2((w_2)^{-1}\circ w_2))=(\One,\One)=\One$. This completes the proof that $\L$ forms a partial group with the product and inversion defined above.
\end{proof}

\begin{definition}
 We call the partial group $\L_1\times\L_2$ constructed above the \textit{(external) direct product} of the partial groups $\L_1$ and $\L_2$.
\end{definition}

\begin{remark}\label{DirectProductPartialGroupIso}
For $i=1,2$ let $\hat{\L}_i$ be a partial group and let $\beta_i\colon \L_i\rightarrow \hat{\L}_i$ be an isomorphism of partial groups. Let $\hat{\L}_1\times\hat{\L}_2$ be the external direct product of $\hat{\L}_1$ and $\hat{\L}_2$. Then the map $\beta\colon \L_1\times\L_2\rightarrow \hat{\L}_1\times\hat{\L}_2$ with $(f,g)\mapsto (f\beta_1,g\beta_2)$ is an isomorphism of partial groups. 
\end{remark}

\begin{proof}
 For $i=1,2$ let $\hat{\Pi}_i\colon\hat{\D}_i\rightarrow \hat{\L}_i$ be the partial product on $\hat{\L}_i$. Write $\hat{\Pi}\colon\hat{\D}\rightarrow \hat{\L}_1\times\hat{\L}_2$ for the partial product on $\hat{\L}_1\times\hat{\L}_2$. Let  $v\in\W(\L_1\times\L_2)$. Observe that, for every $i=1,2$, $(v\beta^*)_i=v_i\beta_i^*$. So using that $\beta_i$ is an isomorphism for each $i=1,2$, we obtain the following equivalence for each $v\in\W(\L)$:
\begin{eqnarray*}
 v\in\D&\Longleftrightarrow& v_i\in\D_i\mbox{ for each }i=1,2\\
&\Longleftrightarrow& v_i\beta_i^*\in\hat{\D}_i\mbox{ for each }i=1,2\\
&\Longleftrightarrow& (v\beta^*)_i\in\hat{\D}_i\mbox{ for each }i=1,2\\
&\Longleftrightarrow& v\beta^*\in\hat{\D}.\\
\end{eqnarray*}
 Hence, $\D\beta^*=\hat{\D}$. Moreover, if $v\in\D$ then $\hat{\Pi}(v\beta^*)=(\hat{\Pi}_1((v\beta^*)_1),\hat{\Pi}_2((v\beta^*)_2))=(\hat{\Pi}_1(v_1\beta_1^*),\hat{\Pi}_2(v_2\beta_2^*))=((\Pi_1(v_1))\beta_1,(\Pi_2(v_2))\beta_2)=
(\Pi_1(v_1),\Pi_2(v_2))\beta=(\Pi(v))\beta$. Clearly $\beta$ is a bijection, so the assertion follows.
\end{proof}

\begin{lemma}\label{DirectSubgroups}
Let $\H_i$ be a partial subgroup of $\L_i$ for $i=1,2$. Then $\H_1\times \H_2$ is a partial subgroup of $\L=\L_1\times\L_2$. If $\H_1$ and $\H_2$ are subgroups of $\L_1$ and $\L_2$ respectively, then $\H_1\times \H_2$ is a subgroup of $\L$ which, regarded as binary group, coincides with the direct product of the (binary) groups $\H_1$ and $\H_2$.  
\end{lemma}

\begin{proof}
Let $f=(f_1,f_2)\in\H_1\times \H_2$ with $f_i\in\H_i$ for $i=1,2$. As $\H_i$ is a partial subgroup, we have $f_i^{-1}\in\H_i$ for $i=1,2$  and thus $f^{-1}=(f_1^{-1},f_2^{-1})\in\H_1\times\H_2$. Let now $w\in\W(\H_1\times\H_2)\cap \D$. Then $w_i\in\W(\H_i)$ for $i=1,2$ and, by definition of $\D$, $w_i\in\D_i$. Hence, $\Pi_i(w_i)\in\H_i$ as $\H_i$ is a partial subgroup for $i=1,2$. Thus $\Pi(w)=(\Pi_1(w_1),\Pi_2(w_2))\in\H_1\times\H_2$. This proves that $\H_1\times\H_2$ is a partial subgroup of $\L$. 

\smallskip

Assume now that $\H_i$ is a subgroup of $\L_i$ for $i=1,2$. Then $\W(\H_i)\subseteq\D_i$ for $i=1,2$. So if $v\in\W(\H_1\times\H_2)$, we have $v_i\in\W(\H_i)\subseteq\D_i$ for $i=1,2$. By definition of $\D$, this implies $v\in\D$ proving $\W(\H_1\times\H_2)\subseteq\D$. So $\H_1\times\H_2$ is a subgroup of $\L$. It follows from the definition of $\L$ and of the direct product of groups that the subgroup $\H_1\times\H_2$ regarded as a binary group coincides with the direct product of the groups $\H_1$ and $\H_2$.
\end{proof}

Let 
\[\pi_i\colon \L\rightarrow \L_i\]
be the projection map for $i=1,2$. This means $\pi_1((f_1,f_2))=f_1$ and $\pi_2((f_1,f_2))=f_2$. 

\begin{lemma}\label{DirectProductsLocalitiesProjections}
 For $i=1,2$, $\pi_i$ is a homomorphism of partial groups. In particular, for any subgroup $\H$ of $\L$, $\H\pi_i$ is a subgroup of $\L_i$ for $i=1,2$. Moreover,  $\H\pi_i$ is a $p$-subgroup of $\L_i$ if $\H$ is a $p$-subgroup of $\L$.
\end{lemma}

\begin{proof}
Let $i\in\{1,2\}$. By \cite[Lemma~1.15]{Chermak:2015}, the image of a subgroup under a homomorphism of partial groups is a subgroup again. So if $\pi_i$ is a homomorphism of partial groups and $\H$ a subgroup of $\L$, then $\H\pi_i$ is a subgroup of $\L_i$ and one observes that $\pi_i|_{\H}\colon \H\rightarrow \H\pi_i$ is a homomorphism of groups. In particular, if $\H$ is a $p$-subgroup of $\L$ then $\H\pi_i$ is a $p$-subgroup. Hence, it is sufficient to prove that $\pi_i$ is a homomorphism of partial groups. That is we need to show that $\D\pi_i^*\subseteq\D_i$ and $(\Pi(w))\pi_i=\Pi_i(w\pi_i^*)$ for all $w\in\D$. Let $w\in\D$. Then $w\pi_i^*=w_i\in\D_i$ by definition of $\D$. Note that $\Pi(w)_i=\Pi_i(w_i)$ as $\Pi(w)=(\Pi_1(w_1),\Pi_2(w_2))$. So $(\Pi(w))\pi_i=\Pi(w)_i=\Pi_i(w_i)=\Pi_i(w\pi_i^*)$. This shows the assertion.
\end{proof}

Define now maps $\iota_1\colon \L_1\rightarrow \L,f\mapsto (f,\One)$ and $\iota_2\colon \L_2\mapsto \L,f\mapsto (\One,f)$. We call $\iota_i$ the inclusion map $\L_i\rightarrow\L$.

\begin{remark}\label{IotaRemark}
 Let $f_i\in\L_i$ for $i=1,2$. Then $((f_1\iota_1),(f_2\iota_2))\in\D$ and $(f_1,f_2)=\Pi((f_1\iota_1),(f_2\iota_2))$.
\end{remark}

\begin{proof}
Let $v=((f_1\iota_1),(f_2\iota_2))=((f_1,\One),(\One,f_2))$. Then $v_1=(f_1,\One)\in\D_1$, $\Pi_1(v_1)=f_1$, $v_2=(\One,f_2)\in\D_2$ and $\Pi_2(v_2)=f_2$ by the axioms of a partial group and by Remark~\ref{Ones}. Hence, $v\in\D$ and $\Pi(v)=(\Pi_1(v_1),\Pi_2(v_2))=(f_1,f_2)$.  
\end{proof}

\begin{lemma}\label{DirectProductsLocalitiesInclusions}
Let $i\in\{1,2\}$. Then the following hold:
\begin{itemize}
 \item [(a)] The subset $\L_i\iota_i$ of $\L=\L_1\times\L_2$ is a partial normal subgroup of $\L$.
 \item [(b)] The map $\iota_i$ is an injective homomorphisms of partial groups which induces an isomorphism of partial groups from $\L_i$ to the partial subgroup $\L_i\iota_i$ of $\L$.  
 \item [(c)] For any partial subgroup $\H$ of $\L_i$, $\H\iota_i$ is a partial subgroup of $\L$. 
\end{itemize} 
\end{lemma}

\begin{proof}
Note that $\L_i\iota_i=\ker(\pi_{3-i})$ for $i=1,2$. By \cite[Lemma~1.14]{Chermak:2015}, the kernel of a homomorphism of partial groups is a partial normal subgroup. Hence, (a) holds.

\smallskip

We prove (b) only for $i=1$, as the proof for $i=2$ is analogous. Let $w=(f_1,\dots,f_n)\in\D_1$ and set $u:=w\iota_1^*=((f_1,\One),\dots,(f_n,\One))$. So $u_1=w\in\D_1$ by assumption. Moreover, $u_2=(\One,\dots,\One)\in\D_2$ and $\Pi_2(u_2)=\One$ by Remark~\ref{Ones}. Hence, $u\in\D$ and $\Pi(w\iota_1^*)=\Pi(u)=(\Pi_1(u_1),\Pi_2(u_2))=(\Pi_1(w),\One)=(\Pi_1(w))\iota_1$. This shows that $\iota_1$ is a homomorphism of partial groups. We regard now $\L_1\iota_1$ as a partial group with product $\Pi|_{\D'}$ where $\D'=\D\cap\W(\L_1\iota_1)$. The properties we proved so far imply that $\D_1\iota_1^*\subseteq\D'$ and that the map $\L_1\rightarrow\L_1\iota_1$ induced by $\iota_1$ is a homomorphism of partial groups. Clearly, $\iota_1$ is injective, so it remains to prove that $\D'\subseteq \D_1\iota_1^*$. Let $u\in \D'$. As $\D'\subseteq\W(\L_1\iota_1)$, $u$ is of the form $u=(f_1\iota_1,\dots,f_n\iota_1)=((f_1,\One),\dots,(f_n,\One))$ with $f_1,\dots,f_n\in\L_1$. Set $v:=(f_1,\dots,f_n)$. As $u\in\D$, we have $v=u_1\in\D_1$. Moreover, $u=(f_1\iota_1,\dots,f_n\iota_1)=v\iota_1^*$. Hence, $u\in\D_1\iota_1^*$. This proves $\D'\subseteq\D_1\iota_1^*$ and completes the proof of (b).

\smallskip

Let $\H$ be a partial subgroup of $\L_i$ for some $i=1,2$. 
By Remark~\ref{IsomorphismOfPartialGroups}, an isomorphism of partial groups maps partial subgroups to partial subgroups. So by (b), $\H\iota_i$ is a partial subgroup of $\L_i\iota_i$. By (a), $\L_i\iota_i$ is a partial subgroup of $\L$. A partial subgroup of a partial subgroup is a partial subgroup again by \cite[Lemma~1.8(a)]{Chermak:2015}. Hence, $\H\iota_i$ is a partial subgroup of $\L$. This proves (c). 

\end{proof}

\begin{remark}\label{ConjugateDirectProduct}
Let $f_1\in\L_1$ and $f_2\in\L_2$ and set $f=(f_1,f_2)$.
\begin{itemize}
 \item [(a)] We have $\D(f)=\D_1(f_1)\times\D_2(f_2)$, where $\D_i(f_i)$ is formed inside of $\L_i$ for $i=1,2$, and $\D(f)$ is formed inside of $\L=\L_1\times\L_2$.
 \item [(b)] If $g_i\in\D_i(f_i)$ for $i=1,2$, then $(g_1^{f_1},g_2^{f_2})=(g_1,g_2)^f$. Similarly, given 
 $P_i\subseteq \D_i(f_i)$ for $i=1,2$, we have $(P_1\times P_2)^f=P_1^{f_1}\times P_2^{f_2}$.
 \item [(c)] Let $S_i\subseteq\L_i$ for $i=1,2$ and $S=S_1\times S_2\subseteq\L$.\footnote{By $S_1\times S_2$ we just mean the set theoretic product of $S_1$ and $S_2$ here} Then $S_f=(S_1)_{f_1}\times (S_2)_{f_2}$, where $(S_i)_{f_i}$ is formed inside of $\L_i$ for $i=1,2$, and $S_f$ is formed inside of $\L$.
\end{itemize}
\end{remark}

\begin{proof}
 Let $g=(g_1,g_2)\in\L$ with $g_i\in\L_i$ for $i=1,2$. We have $g\in\D(f)$ if and only if $v=((f_1^{-1},f_2^{-1}),(g_1,g_2),(f_1,f_2))=(f^{-1},g^{-1},f)\in\D$. By definition of $\D$, this is the case if and only if $(f_i^{-1},g_i,f_i)=v_i\in\D_i$ for $i=1,2$, i.e. if and only if $g_i\in\D_i(f_i)$ for $i=1,2$. This shows $\D(f)=\D_1(f_1)\times\D_2(f_2)$ proving (a). Moreover, $g^f=\Pi(v)=(\Pi_1(v_1),\Pi_2(v_2))=(\Pi_1(f_1^{-1},g_1,f_1),\Pi_2(f_2^{-1},g_2,f_2))=(g_1^{f_1},g_2^{f_2})$. This implies (b). 

\smallskip

Let now $s=(s_1,s_2)\in S_1\times S_2$ with $s_i\in S_i$ for $i=1,2$. Then $s\in S_f$ if and only if $s\in\D(f)$ and $s^f\in S$. By (a) and (b), the latter condition is true if and only if $s_i\in\D_i(f_i)$ for $i=1,2$ and $(s_1^{f_1},s_2^{f_2})=s^f\in S=S_1\times S_2$. This is the case if and only if $s_i\in\D_i(f_i)$ and $s_i^{f_i}\in S_i$ for $i=1,2$, i.e. if and only if $s_i\in (S_i)_{f_i}$ for $i=1,2$. This shows $S_f=(S_1)_{f_1}\times (S_2)_{f_2}$. 
\end{proof}

\begin{lemma}\label{DirectProductCentre}
We have
 \[Z(\L_1\times\L_2)=Z(\L_1)\times Z(\L_2)=Z(\L_1\iota_1)Z(\L_2\iota_2).\]
\end{lemma}

\begin{proof}
Recall $\L=\L_1\times\L_2$. By Remark~\ref{IotaRemark}, it is sufficient to show that $Z(\L)=Z(\L_1)\times Z(\L_2)$. Let $f=(f_1,f_2)\in\L$. Then $f\in Z(\L)$ if and only if $f\in\D(g)$ and $f^g=f$ for all $g\in \L$. By Lemma~\ref{ConjugateDirectProduct}(a),(b), this is the case if and only if $f\in\D(g_1)\times\D(g_2)$ and $(f_1^{g_1},f_2^{g_2})=f^g=f$ for all $g=(g_1,g_2)\in\L$. This is equivalent to $f_i\in\D(g_i)$ and $f_i^{g_i}=f_i$ for all $i=1,2$ and all $g_i\in\L_i$. Hence, $f\in Z(\L)$ if and only if $f_i\in Z(\L_i)$ for $i=1,2$. This implies the assertion.
\end{proof}

\section{External direct and central products of localities}\label{SectionDirectProductExternalLocalities}

For $i=1,2$ let $(\L_i,\Delta_i,S_i)$ be a locality. As in the previous section, $\L_i$ is here a partial group with product $\Pi_i\colon \D_i\rightarrow \L_i$ and inversion map $\L_i\rightarrow \L_i,f\mapsto f^{-1}$. Let $\L=\L_1\times \L_2$ be the partial group we constructed in the previous section, and let $\pi_i\colon\L\rightarrow\L_i$ be the projection map for $i=1,2$. Recall that, by Lemma~\ref{DirectSubgroups}, $P_1\times P_2$ is a subgroup of $\L$ for all $P_1\in\Delta_1$ and $P_2$ in $\Delta_2$. Set $S:=S_1\times S_2$ and let 
\[\Delta=\Delta_1*\Delta_2\]
be the set of subgroups of $S$ containing a subgroup of the form $P_1\times P_2$ with $P_i\in\Delta_i$ for $i=1,2$. We will show that $(\L,\Delta,S)$ is a locality.

\begin{lemma}\label{DirectProductIsLocality}
The triple $(\L,\Delta,S)=(\L_1\times\L_2,\Delta_1*\Delta_2,S_1\times S_2)$ is a locality. Moreover, we have $\F_S(\L)=\F_{S_1}(\L_1)\times \F_{S_2}(\L_2)$.
\end{lemma}

\begin{proof}
In this proof, Remark~\ref{ConjugateDirectProduct} is used frequently, most of the time without reference. We first show that $(\L,\Delta,S)$ is a locality. As $\L_1$ and $\L_2$ are finite as sets, $\L$ is clearly also finite as a set. By Lemma~\ref{DirectSubgroups}, $S=S_1\times S_2$ is a subgroup of $\L$ which, regarded as a binary group, coincides with the direct product of the groups $S_1$ and $S_2$. So $S_1\times S_2$ is a $p$-subgroup of $\L$. Let now $T$ be a $p$-subgroup of $\L$ containing $S=S_1\times S_2$. Then by Lemma~\ref{DirectProductsLocalitiesProjections}, $T\pi_i$ is a $p$-subgroup of $\L_i$ for $i=1,2$. As $S_i\leq T\pi_i$ and $S_i$ is a maximal $p$-subgroup of $\L_i$ for $i=1,2$, it follows $T\pi_i=S_i$ and thus $T=S_1\times S_2$. This shows that $S=S_1\times S_2$ is a maximal $p$-subgroup of $\L$, so (L1) holds.

\smallskip

Let $v=((f_1,g_1),(f_2,g_2),\dots,(f_n,g_n))\in\W(\L)$ (where $f_j\in\L_1$ and $g_j\in\L_2$ for $j=1,\dots,n$). Recall that $\D_i=\D_{\Delta_i}$ for $i=1,2$ (where $\D_{\Delta_i}$ is formed inside of $\L_i$), since $(\L_i,\Delta_i,S_i)$ is a locality. Using this property and Remark~\ref{ConjugateDirectProduct}, we get the following equivalence:
\begin{eqnarray*}
 & & v\in\D\\
&\Longleftrightarrow& v_1=(f_1,\dots,f_n)\in\D_1\mbox{ and }v_2=(g_1,\dots,g_n)\in\D_2\\
&\Longleftrightarrow& \mbox{There exist }P_0,\dots,P_n\in\Delta_1\mbox{ and }Q_0,\dots,Q_n\in\Delta_2\mbox{ such that}\\
& & \mbox{for }j=1,\dots,n,\;P_{j-1}\subseteq S_{f_j},\;P_{j-1}^{f_j}=P_j,\;Q_{j-1}\subseteq S_{g_j},\;Q_{j-1}^{g_j}=Q_j,\\
&\Longleftrightarrow& \mbox{There exist }P_0,\dots,P_n\in\Delta_1\mbox{ and }Q_0,\dots,Q_n\in\Delta_2\mbox{ such that}\\
& & \mbox{for }j=1,\dots,n,\;(P_{j-1}\times Q_{j-1})\subseteq S_{(f_j,g_j)}\mbox{ and }(P_{j-1}\times Q_{j-1})^{(f_j,g_j)}=P_j\times Q_j\\
&\Longleftrightarrow& \mbox{There exist }P_0,\dots,P_n\in\Delta_1\mbox{ and }Q_0,\dots,Q_n\in\Delta_2\mbox{ such that }\\
& & v\in\D_{\Delta}\mbox{ via }P_0\times Q_0,\dots,P_n\times Q_n.
\end{eqnarray*}
In particular, $v\in\D$ implies $v\in\D_{\Delta}$. Suppose now $v\in\D_{\Delta}$. Then there exist $X_0,\dots,X_n\in\Delta$ such that $v\in\D_{\Delta}$ via $X_0,\dots,X_n$. By definition of $\Delta$, there exist $P_0\in\Delta_1$ and $Q_0\in\Delta_2$ such that $P_0\times Q_0\leq X_0$. Define $P_j$ and $Q_j$ recursively by $P_j=P_{j-1}^{f_j}$ and $Q_j=Q_{j-1}^{g_j}$. As $X_{j-1}\subseteq S_{(f_j,g_j)}=(S_1)_{f_j}\times (S_2)_{g_j}$ and $X_{j-1}^{(f_j,g_j)}=X_j$, an easy induction argument shows that, for $j=1,\dots,n$, $P_j$ and $Q_j$ are well-defined and $P_j\times Q_j=(P_{j-1}\times Q_{j-1})^{(f_j,g_j)}\leq X_j$. As $\L_1$ and $\L_2$ are localities, $P_j\in\Delta_1$ and $Q_j\in\Delta_2$ for $j=1,\dots,n$ by Lemma~\ref{LocalitiesProp}(b). In particular, $P_j\times Q_j\in\Delta$ for $j=1,\dots,n$. Hence, $v\in\D_{\Delta}$ via $P_0\times Q_0,\dots,P_n\times Q_n$. By the above equivalence, this means $v\in\D$.  Thus, we have shown that $\D=\D_\Delta$, i.e. (L2) holds. 

\smallskip

It remains to prove (L3). Let $X\in\Delta$ and $g=(g_1,g_2)\in\L$ such that $X\subseteq S_g$. Let $X^g\leq Y\leq S$. We need to show that $Y\in\Delta$. As $X\in\Delta$, there exist $P_i\in\Delta_i$ for $i=1,2$ such that $P_1\times P_2\leq X$. Then $P_1\times P_2\subseteq S_g=(S_1)_{g_1}\times (S_2)_{g_2}$ and thus $P_i\leq (S_i)_{g_i}$ for $i=1,2$. Thus, as $\L_i$ is a locality, we have $P_i^{g_i}\in\Delta_i$ for $i=1,2$ by Lemma~\ref{LocalitiesProp}(b). As $P_1^{g_1}\times P_2^{g_2}=(P_1\times P_2)^g\subseteq X^g\leq Y\leq S$, it follows now from the definition of $\Delta$ that $Y\in\Delta$. This shows (L3) and completes the proof that $(\L,\Delta,S)$ is a locality.

\smallskip

Set $\F_i:=\F_{S_i}(\L_i)$ for $i=1,2$. It remains to show that $\F_S(\L)=\F_1\times \F_2$. The fusion system $\F_S(\L)$ is generated by the maps $c_g\colon S_g\rightarrow S$ with $g\in\L$. Take $g=(g_1,g_2)\in\L$. We have $S_g=(S_1)_{g_1}\times (S_2)_{g_2}$. Moreover, for any $s=(s_1,s_2)\in S_g$,  $sc_g=s^g=(s_1^{g_1},s_2^{g_2})=(s_1c_{g_1},s_2c_{g_2})$. So using the notation from Section~\ref{DirectProductFusionSystemsSection}, we have $c_g=c_{g_1}\times c_{g_2}\in\Hom_{\F_1\times \F_2}(S_g,S)$. This shows $\F_S(\L)\subseteq\F_1\times \F_2$. The fusion system $\F_1\times \F_2$ is generated by maps $\phi_1\times\phi_2$ where $\phi_i$ is a morphism in $\F_i$. So let $P_i,Q_i\leq S_i$ and $\phi_i\in\Hom_{\F_i}(P_i,Q_i)$ for $i=1,2$. We need to show that $\phi_1\times \phi_2\colon P_1\times P_2\rightarrow Q_1\times Q_2$ is a morphism in $\F_S(\L)$. By definition of $\F_1$, there exist elements $f_1,\dots,f_n\in\L_1$ and subgroups $P_1=X_0,\dots,X_n$ of $S$ such that, for $j=1,\dots,n$, $X_{j-1}^{f_j}=X_j$ and $\phi_1=c_{f_1}|_{X_0}\circ \dots \circ c_{f_n}|_{X_{n-1}}$. Similarly, by definition of $\F_2$, there exist elements $g_1,\dots,g_m\in\L_2$ and subgroups $P_2=Y_0,\dots,Y_m$ of $S$ such that, for $j=1,\dots,m$, $Y_{j-1}^{f_j}=Y_j$ and $\phi_2=c_{g_1}|_{Y_0}\circ \dots \circ c_{g_m}|_{Y_{m-1}}$. Inserting conjugation maps by the identity element if necessary, we may assume $m=n$. Then setting $Z_j=X_j\times Y_j$ for $j=0,\dots,n$ and $h_j=(f_j,g_j)$ for $j=1,\dots,n$, we have $Z_{j-1}^{h_j}=Z_j$ for $j=1,\dots,n$, and $\phi_1\times \phi_2=c_{h_1}|_{Z_0}\circ c_{h_2}|_{Z_1}\circ \dots\circ c_{h_n}|_{Z_n}$. So $\phi_1\times\phi_2$ is a morphism in $\F_S(\L)$. This shows $\F_S(\L)=\F_1\times\F_2$ as required and completes the proof.
\end{proof}

\begin{definition}
 We call $(\L,\Delta,S)=(\L_1\times\L_2,\Delta_1*\Delta_2,S_1\times S_2)$ the \textit{(external) direct product} of $(\L_1,\Delta_1,S_1)$ and $(\L_2,\Delta_2,S_2)$.
\end{definition}

As before let $\iota_i\colon\L_i\rightarrow\L$ be the inclusion map.

\begin{lemma}\label{DirectProductLiSublocality}
 For $i=1,2$, $(\L_i\iota_i,\Delta_i\iota_i,S_i\iota_i)$ is a sublocality of $(\L,\Delta,S)$ and $\F_{S_i\iota_i}(\L_i\iota_i)$ is the canonical image of $\F_{S_i}(\L_i)$ in $\F_{S_1}(\L_1)\times\F_{S_2}(\L_2)=\F_S(\L)$.
\end{lemma}

\begin{proof}
Recall that, by Lemma~\ref{DirectProductIsLocality}, $\F_S(\L)=\F_{S_1}(\L_1)\times \F_{S_2}(\L_2)$. Let $i\in\{1,2\}$. Set $\F_i=\F_{S_i}(\L_i)$, $\hat{\L}_i=\L_i\iota_i$, $\hat{\Delta}_i=\Delta_i\iota_i$ and $\hat{S}_i=S_i\iota_i$. Lemma~\ref{DirectProductsLocalitiesInclusions}(b) gives that $\iota_i$ is a homomorphism of partial groups and $\D_i\iota_i^*=\D\cap \W(\hat{\L}_i)$. Moreover, $S_i\iota_i\subseteq S_1\times S_2=S$. Since $(\L_i,\Delta_i,S_i)$ is a sublocality of itself, it follows from Lemma~\ref{SublocalityUnderPartialHom} that $(\L_i\iota_i,\Delta_i\iota_i,S_i\iota_i)$ is a sublocality of $(\L,\Delta,S)$. Moreover, the restriction of $\iota_i$ to a map $\L_i\rightarrow \L_i\iota_i$ is a projection of localities from $(\L_i,\Delta_i,S_i)$ to $(\L_i\iota_i,\Delta_i\iota_i,S_i\iota_i)$. Because of the latter property, it follows from \cite[Theorem~5.7(b)]{Henke:2015} that $(\iota_i)|_{S_i}\colon S_i\rightarrow S_i\iota_i$ induces an epimorphism from $\F_{S_i}(\L_i)$ to $\F_{S_i\iota_i}(\L_i\iota_i)$, i.e. $\F_{S_i}(\L_i)(\iota_i)|_{S_i}= \F_{S_i\iota_i}(\L_i\iota_i)$. Note that $(\iota_i)|_{S_i}$ is the canonical inclusion map $S_i\rightarrow S_1\times S_2$ which induces a morphism of fusion systems from $\F_{S_i}(\L_i)$ to $\F_{S_1}(\L_1)\times \F_{S_2}(\L_2)=\F_S(\L)$. The canonical isomorphic image of $\F_{S_i}(\L_i)$ in $\F_{S_1}(\L_1)\times \F_{S_2}(\L_2)$ is by definition the image under this morphism and equals thus $\F_{S_i}(\L_i)(\iota_i)|_{S_i}=\F_{S_i\iota_i}(\L_i\iota_i)$. This shows the assertion.
\end{proof}

Our next goal now will be to show that $(\L,\Delta,S)$ is objective characteristic $p$ if and only if $(\L_i,\Delta_i,S_i)$ is of objective characteristic $p$ for each $i=1,2$. We will need the following elementary group theoretical lemma:

\begin{lemma} \label{GroupsDirectProductCharp}
 Let $G_1$ and $G_2$ be finite groups. Then $G_1\times G_2$ is of characteristic $p$ if and only if $G_1$ and $G_2$ are of characteristic $p$.
\end{lemma}

\begin{proof}
Set $G:=G_1\times G_2$ and observe that $O_p(G)=O_p(G_1)\times O_p(G_2)$. If $G$ has characteristic $p$ then $C_{G_i}(O_p(G_i))\leq C_G(O_p(G))\cap G_i\leq O_p(G)\cap G_i=O_p(G_i)$ and $G_i$ has characteristic $p$ for $i=1,2$. If $G_1$ and $G_2$ have characteristic $p$ then $C_G(O_p(G))=C_G(O_p(G_1)\times O_p(G_2))=C_{G_1}(O_p(G_1))\times C_{G_2}(O_p(G_2))\leq O_p(G_1)\times O_p(G_2)=O_p(G)$. Hence, $G$ is of characteristic $p$.
\end{proof}

\begin{lemma}\label{DirectProductsLocalitiesNormalizers}
 As before let $(\L,\Delta,S)$ be the external direct product of $(\L_1,\Delta_1,S_1)$ and $(\L_2,\Delta_2,S_2)$. Let $P\leq S$ and set $P_i:=P\pi_i$. Then the following hold: 
\begin{itemize}
\item [(a)] We have $N_\L(P)\subseteq N_{\L_1}(P_1)\times N_{\L_2}(P_2)$.
\item [(b)] If $P=P_1\times P_2$ then  $N_\L(P)=N_{\L_1}(P_1)\times N_{\L_2}(P_2)$ (as a set).
\item [(c)] Suppose $P\in\Delta$. Then $P_i\in\Delta_i$ for $i=1,2$ and in particular $P_1\times P_2\in\Delta$. If the groups $N_{\L_1}(P_1)$ and $N_{\L_2}(P_2)$ are of characteristic $p$ then $N_\L(P)$ is of characteristic $p$. 
\item [(d)] If $P=P_1\times P_2\in\Delta$ then $N_\L(P)$ is of characteristic $p$ if and only if $N_{\L_i}(P_i)$ is of characteristic $p$ for $i=1,2$. 
\end{itemize}
\end{lemma}

\begin{proof}
Let $f=(f_1,f_2)\in\L$ with $f_i\in\L_i$ for $i=1,2$. For the proof of (a), suppose  $f\in N_\L(P)$. Let $x_1\in P_1$. Since $P_1$ is the projection of $P$ to $S_1$, there exists $x_2\in P_2$ such that $x=(x_1,x_2)\in P$. Using Remark~\ref{ConjugateDirectProduct}(a),(b), we get $x_i\in\D(f_i)$ for $i=1,2$ and $(x_1^{f_1},x_2^{f_2})=x^f\in P$ as $f\in N_\L(P)$. Hence, $x_1^{f_1}=x^f\pi_1\in P_1$ proving $f_1\in N_{\L_1}(P_1)$. Similarly, one shows $f_2\in N_{\L_2}(P_2)$. So if $f\in N_\L(P)$ then $f=(f_1,f_2)\in N_{\L_1}(P_1)\times N_{\L_2}(P_2)$. This proves (a). 

\smallskip

For the proof of (b) assume $P=P_1\times P_2$ and $f_i\in N_{\L_i}(P_i)$ for $i=1,2$. By (a), it remains to prove that $f\in N_\L(P)$. By Remark~\ref{ConjugateDirectProduct}(a),(b), $P=P_1\times P_1\subseteq\D_1(f_1)\times \D_2(f_2)=\D(f)$ and $P^f=P_1^{f_1}\times P_2^{f_2}=P_1\times P_2=P$. So $f\in N_\L(P)$ as required. This proves (b).

\smallskip

Let now $P\in\Delta$ be arbitrary. Then, by definition of $\Delta$, there exist $Q_i\in\Delta_i$ for $i=1,2$ such that $Q_1\times Q_2\leq P$. Then for $i=1,2$, we have $Q_i\leq P_i$ and thus $P_i\in\Delta_i$, as $\Delta_i$ is closed under taking overgroups in $S_i$. By definition of $\Delta$, it follows $P_1\times P_2\in\Delta$.

\smallskip

By (a) and (b), $H:=N_\L(P)\subseteq G=N_\L(P_1\times P_2)=N_{\L_1}(P_1)\times N_{\L_2}(P_2)$. Since the normalizer of an object in a locality is a finite group by Lemma~\ref{LocalitiesProp}(a), $H$, $G$, $N_{\L_1}(P_1)$ and $N_{\L_2}(P_2)$ are finite groups. By Lemma~\ref{DirectSubgroups}, $G$ regarded as a binary group coincides with the direct product of the binary groups $N_{\L_1}(P_1)$ and $N_{\L_2}(P_2)$. 

\smallskip

If $N_{\L_1}(P_1)$ and $N_{\L_2}(P_2)$ are of characteristic $p$ then $G$ is of characteristic $p$ by Lemma~\ref{GroupsDirectProductCharp}. By \cite[Lemma~1.2(c)]{MS:2012b}, every $p$-local subgroup of a group of characteristic $p$ is of characteristic $p$. Hence, $H=N_G(P)$ is of characteristic $p$ if $G$ is of characteristic $p$. This proves (d). Suppose now $P=P_1\times P_2$. Then $G=H$ and thus (e) follows from Lemma~\ref{GroupsDirectProductCharp}. 
\end{proof}

\begin{lemma}\label{DirectProductObjectiveCharp}
 The locality $(\L,\Delta,S)=(\L_1\times\L_2,\Delta_1*\Delta_2,S_1\times S_2)$ is of objective characteristic $p$ if and only if $(\L_i,\Delta_i,S_i)$ is of objective characteristic $p$ for each $i=1,2$.
\end{lemma}

\begin{proof}
If $(\L_i,\Delta_i,S_i)$ is of objective characteristic $p$ for $i=1,2$, then it follows from Lemma~\ref{DirectProductsLocalitiesNormalizers}(c) that $N_\L(P)$ is of characteristic $p$ for any $P\in\Delta$. So $(\L,\Delta,S)$ is of objective characteristic $p$ if  $(\L_i,\Delta_i,S_i)$ is of objective characteristic $p$ for $i=1,2$. Suppose now $(\L,\Delta,S)$ is of objective characteristic $p$. Let $P_i\in\Delta_i$ for $i=1,2$. We need to see that $N_{\L_i}(P_i)$ is of characteristic $p$ for $i=1,2$. Setting $P=P_1\times P_2$ this follows from Lemma~\ref{DirectProductsLocalitiesNormalizers}(d).  
\end{proof}

\begin{lemma}\label{DirectProductLinkingLocality}
 The locality $(\L,\Delta,S)=(\L_1\times\L_2,\Delta_1*\Delta_2,S_1\times S_2)$ is a linking locality if and only if $(\L_i,\Delta_i,S_i)$ is a linking locality for each $i=1,2$.
\end{lemma}

\begin{proof}
Set $\F=\F_S(\L)$ and $\F_i=\F_{S_i}(\L_i)$ for $i=1,2$. By Lemma~\ref{DirectProductObjectiveCharp}, it is sufficient to show that $\F^{cr}\subseteq\Delta$ if and only if $\F_i^{cr}\subseteq\Delta_i$. Recall that by Lemma~\ref{DirectProductIsLocality}, $\F=\F_1\times\F_2$. So by Lemma~\ref{DirectProductFusionSystems}(c), $\F^{cr}=\{R_1\times R_2\colon R_i\in\F_i^{cr}\mbox{ for }i=1,2\}$. In particular, clearly $\F^{cr}\subseteq\Delta$ if $\F_i^{cr}\subseteq \Delta_i$. Assume now $\F^{cr}\subseteq\Delta$. Let $R_i\in\F_i^{cr}$ for $i=1,2$. Then $R_1\times R_2\in\F^{cr}\subseteq\Delta$. Then by Lemma~\ref{DirectProductsLocalitiesNormalizers}, $R_i\in\Delta_i$ for $i=1,2$. This shows $\F_i^{cr}\subseteq\Delta_i$ for $i=1,2$ provided $\F^{cr}\subseteq\Delta$. Hence, the proof is complete.
\end{proof}

Let $\iota_i\colon \L_i\rightarrow \L$ be the inclusion map for $i=1,2$. Recall from Lemma~\ref{DirectProductCentre} that $Z(\L)=Z(\L_1)\times Z(\L_2)=Z(\L_1\iota_1)Z(\L_2\iota_2)$. Observe also that every subgroup of $Z(\L)$ is a partial normal subgroup of $\L$.

\begin{definition}
Recall that we assume throughout $(\L,\Delta,S)=(\L_1\times\L_2,\Delta_1*\Delta_2,S_1\times S_2)$. Let $Z\leq Z(\L)$ with $Z\cap (\L_i\iota_i)=\{\One\}$ for $i=1,2$. Let $\beta\colon \L\rightarrow\L/Z$ be the canonical  projection map as defined in Subsection~\ref{SubsectionLocalitiesProjections}. Then we call the locality $(\L/Z,\Delta\beta,S\beta)$ the \textit{(external) central product} of the localities $(\L_1,\Delta_1,S_1)$ and $(\L_2,\Delta_2,S_2)$ over $Z$.   
\end{definition}

The reader should note that it is not so clear how one should define external central products of arbitrary partial groups since quotients of partial groups modulo partial normal subgroups are not defined in general.

\begin{lemma}\label{ExternalCentralProductLemma}
 Let $Z\leq Z(\L)$ with $Z\cap (\L_i\iota_i)=\{\One\}$ and let $\beta\colon \L\rightarrow\L/Z$ be the canonical projection so that $(\L/Z,\Delta\beta,S\beta)$ is the external central product of the localities $(\L_1,\Delta_1,S_1)$ and $(\L_2,\Delta_2,S_2)$ over $Z$. Set $\F_i=\F_{S_i}(\L_i)$ for $i=1,2$.
\begin{itemize}
\item [(a)] The localities $(\L_1,\Delta_1,S_1)$ and $(\L_2,\Delta_2,S_2)$ are of objective characteristic $p$ if and only if $Z\leq S$ and the central product $(\L/Z,\Delta\beta,S\beta)$ is of objective characteristic $p$.
\item [(b)] The localities $(\L_1,\Delta_1,S_1)$ and $(\L_2,\Delta_2,S_2)$ are linking localities if and only if $Z\leq S$ and the central product $(\L/Z,\Delta\beta,S\beta)$ is a linking locality. 
\item [(c)] If $Z\leq S$ then $Z\leq Z(\F_1\times\F_2)$, $(\F_1\times \F_2)/Z$ is a central product of the fusion systems $\F_1$ and $\F_2$, and $(\L/Z,\Delta\beta,S\beta)$ is a locality over $(\F_1\times\F_2)/Z$.
\end{itemize}
\end{lemma}

\begin{proof}
If $(\L,\Delta,S)$ is of objective characteristic $p$ then $Z\leq C_\L(S)\leq S$. In particular, $Z\leq S$ if $(\L,\Delta,S)$ is a linking locality. Assume from now on that $Z\leq S$. Recall from Lemma~\ref{DirectProductIsLocality} that $\F_S(\L)=\F_1\times \F_2$. In particular, $Z\leq Z(\F_1\times\F_2)$ as $Z\leq Z(\L)$ and $\F_S(\L)$ is generated by the conjugation maps by elements of $\L$.

\smallskip

As $Z\leq Z(\L)\cap S$, \cite[Proposition~9.2]{Henke:2015} gives us the following properties: The locality $(\L/Z,\Delta\beta,S\beta)$ is a locality over $(\F_1\cap \F_2)/Z$; $(\L/Z,\Delta\beta,S\beta)$ is of objective characteristic $p$ if and only if $(\L,\Delta,S)$ is of objective characteristic $p$; and $(\L/Z,\Delta\beta,S\beta)$ is a linking locality if and only if $(\L,\Delta,S)$ is a linking locality. Now (a) and (b) follow from Lemma~\ref{DirectProductObjectiveCharp} and Lemma~\ref{DirectProductLinkingLocality}. We have $Z\cap (S_i\iota_i)\leq Z\cap Z(\L_i\iota_i)=\{\One\}$ and thus $Z\cap (S_i\iota_i)=1$. Hence, $(\F_1\times \F_2)/Z$ is a central product. So (c) holds.
\end{proof}

\section{Internal central and direct products}\label{SectionInternal}

Throughout this section let $\L$ be a partial group with product $\Pi\colon\D\rightarrow\L$. For $i=1,2$, $\L_i$ will always be a partial group with product $\Pi_i\colon\D_i\rightarrow \L_i$. Moreover  
\[\iota_i\colon \L_i\rightarrow \L_1\times \L_2\] 
denotes the inclusion map from $\L_i$ into the external direct product $\L_1\times \L_2$. For $i=1,2$ we set  $\hat{\L_i}:=\L_i\iota_i$, i.e.  $\hat{\L_1}=\{(f,\One)\colon f\in\L_1\}$ and $\hat{\L_2}=\{(\One,g)\colon g\in\L_2\}$. By Lemma~\ref{DirectProductsLocalitiesInclusions}(a),(b), $\hat{\L_1}$ and $\hat{\L_2}$ are partial normal subgroups of $\L_1\times\L_2$, and $\iota_i$ induces an isomorphism $\L_i\rightarrow \hat{\L_i}$. 

\smallskip

Except in Example~\ref{DirectProductPartialGroupExternalInternal} and Example~\ref{CentralProductExternalInternal}, $\L_1$ and $\L_2$ are assumed to be partial subgroups of $\L$, $\D_i:=\D\cap\W(\L_i)$ and $\Pi_i:=\Pi|_{\D_i}$.

\begin{definition}
 We say that $\L$ is the \textit{(internal) central product} of $\L_1$ and $\L_2$ if the following conditions hold:
\begin{itemize}
 \item [(C1)] We have 
\begin{eqnarray*}
\D=\{(\Pi(f_1,g_1),\dots,\Pi(f_n,g_n))\colon (f_1,\dots,f_n)\in\D\cap \W(\L_1),\;(g_1,\dots,g_n)\in\D\cap\W(\L_2),& &\\
(f_j,g_j)\in\D\mbox{ for }j=1,\dots,n\}&.&
\end{eqnarray*} 
 \item [(C2)] If $(f_1,\dots,f_n)\in\D\cap \W(\L_1)$, $(g_1,\dots,g_n)\in\D\cap\W(\L_2)$ and $(f_j,g_j)\in\D$ for $j=1,\dots,n$, then $\Pi(\Pi(f_1,g_1),\dots,\Pi(f_n,g_n))=\Pi(\Pi(f_1,\dots,f_n),\Pi(g_1,\dots,g_n))$. 
\end{itemize}
We call $\L$ the \textit{(internal) direct product} of $\L_1$ and $\L_2$ if $\L$ is the central product of $\L_1$ and $\L_2$ and the following additional property holds: 
\begin{itemize}
 \item [(D)] For any $h\in\L$ there exist unique elements $f\in\L_1$ and $g\in\L_2$ with $(f,g)\in\D$ and $h=\Pi(f,g)$. 
\end{itemize}
\end{definition}

\begin{remark}\label{CentralProductProduct}
If (C1) holds then $\L=\L_1\L_2$. In other words, for every $h\in\L$ there exist elements $f\in\L_1$ and $g\in\L_2$ with $(f,g)\in\D$ and $h=\Pi(f,g)$. So the important part in property (D) is the uniqueness of $f$ and $g$.
\end{remark}

\begin{proof}
Let $h\in\L$. Then $(h)\in\D$ by the axioms of a partial group. So by (C1), there exist $(f)\in\D\cap\W(\L_1)$ and $(g)\in\D\cap\W(\L_2)$ with $(f,g)\in\D$ and $(h)=(\Pi(f,g))$. Then $f\in\L_1$, $g\in\L_2$ and $h=\Pi(f,g)$.
\end{proof}

\begin{lemma}\label{CentralProductCentralizer}
 Suppose that $\L$ is the internal central product of $\L_1$ and $\L_2$. Then $\L_1\subseteq C_\L(\L_2)$ and $\L_2\subseteq C_\L(\L_1)$. In particular, for all $f\in\L_1$ and $g\in\L_2$, we have $(f,g)\in\D$, $(g,f)\in\D$ and $fg=gf$. Moreover, $\L$ is the internal central product of $\L_2$ and $\L_1$.
\end{lemma}

\begin{proof}
 Let $f\in\L_1$ and $g\in\L_2$. We show first that $g\in\D(f)$ and $g^f=g$. By the axioms of a partial group, $(f^{-1},f)\in\D$, $(g)\in\D$, $\Pi(f^{-1},f)=\One$ and $\Pi(g)=g$. So by Remark~\ref{Ones}, $(f^{-1},\One,f)\in\D\cap\W(\L_1)$, $(\One,g,\One)\in\D\cap\W(\L_2)$, $\Pi(f^{-1},\One,f)=\Pi(f^{-1},f)=\One$ and $\Pi(\One,g,\One)=\Pi(g)=g$. A similar argument shows that $(f^{-1},\One)$, $(\One,g)$ and $(f,\One)$ lie in $\D$ and $\Pi(f^{-1},\One)=f^{-1}$, $\Pi(\One,g)=g$ and $\Pi(f,\One)=f$. So by (C1), $(f^{-1},g,f)=(\Pi(f^{-1},\One),\Pi(\One,g),\Pi(f,\One))\in\D$ and by (C2), $g^f=\Pi(f^{-1},g,f)=\Pi(\Pi(f^{-1},\One,f),\Pi(\One,g,\One))=\Pi(\One,g)=g$. This proves $\L_1\subseteq C_\L(\L_2)$. So by Lemma~\ref{Centralizers}, $\L_2\subseteq C_\L(\L_1)$ and, for all $f\in\L_1$ and all $g\in\L_2$, we have $(f,g)\in\D$, $(g,f)\in\D$ and $\Pi(f,g)=\Pi(g,f)$. It follows from the latter property and the definition of an internal central product that $\L$ is the internal central product of $\L_2$ and $\L_1$,
\end{proof}

\begin{prop}\label{InternalCentralProductsPartialGroups}
Consider the map
\[\phi\colon\L_1\times \L_2\rightarrow \L\mbox{ with }(f,g)\mapsto \Pi(f,g)\]
which is well-defined if $(f,g)\in\D$ for all $f\in\L_1$ and $g\in\L_2$. The following hold:
\begin{itemize}
\item [(a)] The map $\phi$ is well-defined and a projection of partial groups if and only if $\L$ is the internal central product of $\L_1$ and $\L_2$. 
\item [(b)] If  $\phi$ is well-defined and a projection of partial groups then \[\ker(\phi)=\{(f,f^{-1})\colon f\in\L_1\cap\L_2\}\leq Z(\L_1\times\L_2)\] and $\ker(\phi)\cap\hat{\L_i}=\{\One\}$.
\item [(c)] The map $\phi$ is well-defined and an isomorphism of partial groups if and only if $\L$ is the internal direct product of $\L_1$ and $\L_2$. 
\item [(d)] If $\phi$ is well-defined then $\hat{\L_i}\phi=\L_i$ and the map $\hat{\L_i}\rightarrow\L_i$ induced by $\phi$ is an isomorphism for $i=1,2$. 
\end{itemize}
\end{prop}

\begin{proof}
By Lemma~\ref{CentralProductCentralizer}, if $\L$ is the central product of $\L_1$ and $\L_2$ then $(f,g)\in\D$ for all $f\in\L_1$ and $g\in\L_2$, i.e. $\phi$ is well-defined. Therefore, we assume in the remainder of the proof that $\phi$ is well-defined. For (a), we will show that $\phi$ is a projection if and only if $\L$ is the central product of $\L_1$ and $\L_2$. 
Write $\Pi'$ for the product on $\L_1\times\L_2$ and $\D'$ for its domain.
Note that 
\begin{eqnarray*}
\D'\phi^*&=&\{(\Pi(f_1,g_1),\dots,\Pi(f_n,g_n))\colon ((f_1,g_1),\dots,(f_n,g_n))\in\D'\}\\
&=&\{(\Pi(f_1,g_1),\dots,\Pi(f_n,g_n))\colon (f_1,\dots,f_n)\in\D\cap\W(\L_1),\;(g_1,\dots,g_n)\in\D\cap\W(\L_2)\}
\end{eqnarray*}
where the first equality follows from the definition of $\phi$ and the second equality follows from the definition of the domain $\D'$ of $\L_1\times\L_2$. Hence, as $(f,g)\in\D$ for all $f\in\L_1$ and all $g\in\L_2$, (C1) holds if and only if $\D=\D'\phi^*$.

\smallskip

Let now $v=((f_1,g_1),\dots,(f_n,g_n))\in\D'$, or equivalently, $(f_1,\dots,f_n)\in\D_1=\D\cap\W(\L_1)$ and $(g_1,\dots,g_n)\in\D_2=\D\cap \W(\L_2)$. We have $v\phi^*=(\Pi(f_1,g_1),\dots,\Pi(f_n,g_n))$ and thus 
\[\Pi(v\phi^*)=\Pi(\Pi(f_1,g_1),\dots,\Pi(f_n,g_n)).\] 
Moreover, $\Pi'(v)=(\Pi_1(f_1,\dots,f_n),\Pi_2(g_1,\dots,g_n))=(\Pi(f_1,\dots,f_n),\Pi(g_1,\dots,g_n))$ by definition of the product $\Pi'$ on $\L_1\times\L_2$. Thus
\[(\Pi'(v))\phi=\Pi(\Pi(f_1,\dots,f_n),\Pi(g_1,\dots,g_n)).\]
Hence, we have $\Pi(v\phi^*)=(\Pi'(v))\phi$ for all $v\in\D'$ if and only if (C2) holds. This proves (a).

\smallskip

For (b) assume that $\phi$ is well-defined and a projection of partial groups (so that $\ker(\phi)$ is well-defined). Clearly, for all $f\in\L_1\cap\L_2$, we have $(f,f^{-1})\phi=\Pi(f,f^{-1})=\One$ and thus $(f,f^{-1})\in\ker(\phi)$. Let now $(f,g)\in\ker(\phi)$ with $f\in\L_1$ and $g\in\L_2$. Then $\Pi(f,g)=(f,g)\phi=\One=\Pi(f,f^{-1})$. Hence, by the left cancellation property \cite[Lemma~1.4(e)]{Chermak:2015}, $g=f^{-1}$. So $g=f^{-1}\in\L_1\cap\L_2$ and thus $f\in\L_1\cap\L_2$. This shows $(f,g)=(f,f^{-1})$ with $f\in\L_1\cap\L_2$. Hence, $\ker(\phi)=\{(f,f^{-1})\colon f\in\L_1\cap\L_2\}$. By Lemma~\ref{CentralProductCentralizer}, $\L_1\cap\L_2\subseteq Z(\L_i)$ for $i=1,2$. So we have $\ker(\phi)\subseteq (\L_1\cap\L_2)\times(\L_1\cap\L_2)\subseteq Z(\L_1)\times Z(\L_2)=Z(\L_1\times\L_2)$ by Lemma~\ref{DirectProductCentre}. Clearly, $\ker(\phi)\cap\hat{\L_i}=\{\One\}$ for $i=1,2$. This shows (b).

\smallskip

Property (D) means that for each $h\in\L$ there exists a unique $(f,g)\in\L_1\times\L_2$ with $(f,g)\phi=\Pi(f,g)=h$, i.e. that $\phi$ is bijective. Hence, (c) follows from (a).

\smallskip

For $i=1,2$, let $\hat{\iota_i}$ be the restriction of $\iota_i$ to a map $\L_i\rightarrow\hat{\L_i}$, which by Lemma~\ref{DirectProductsLocalitiesInclusions}(b) is an isomorphism of partial groups. Thus, $f\hat{\iota_1}=(f,\One)$ for all $f\in\L_1$, and $g\hat{\iota_2}=(\One,g)$ for all $g\in\L_2$. Note that, for all $f\in\L_1$, we have $(f,\One)\phi=\Pi(f,\One)=f$ and, for all $g\in\L_2$, we have $(\One,g)\phi=\Pi(\One,g)=g$. Thus, for $i=1,2$, $\hat{\L_i}\phi=\L_i$ and the map $\hat{\L_i}\rightarrow \L_i$ induced by $\phi$ is the same as $\hat{\iota_i}^{-1}$. By Remark~\ref{IsomorphismOfPartialGroups}(a), the inverse map of an isomorphism of partial groups is an isomorphism of partial groups. Hence, $\phi$ induces an isomorphism of partial groups $\hat{\L_i}\rightarrow\L_i$. 
\end{proof}

\begin{example}\label{DirectProductPartialGroupExternalInternal}
Let $\L_1$ and $\L_2$ be arbitrary partial groups (not necessarily partial subgroups of $\L$). Then the external direct product $\L_1\times\L_2$ is the internal direct product of $\hat{\L_1}$ and $\hat{\L_2}$. 
\end{example}

\begin{proof}
Set $\L=\L_1\times \L_2$. We prove the assertion using Proposition~\ref{InternalCentralProductsPartialGroups}(c) (even though it would also be possible to give a direct proof). So we show that the map $\phi\colon \hat{\L}_1\times\hat{\L}_2\rightarrow \L$ with $(\hat{f},\hat{g})\mapsto \Pi(\hat{f},\hat{g})$ is an isomorphism. Notice that, for all  $\hat{f}\in\hat{\L}_1$ and $\hat{g}\in\hat{\L}_2$, there exist $f\in\L_1$ and $g\in\L_2$ such that $\hat{f}=f\iota_1$ and $\hat{g}=g\iota_2$. Then by Remark~\ref{IotaRemark}, $(\hat{f},\hat{g})\in\D$ and $\Pi(\hat{f},\hat{g})=(f,g)$. So $\phi$ is well-defined and the inverse of the map $\L_1\times\L_2\rightarrow\hat{\L}_1\times\hat{\L}_2,(f,g)\mapsto (f\iota_1,g\iota_2)$ which is an isomorphism of partial groups by Remark~\ref{DirectProductPartialGroupIso} and Lemma~\ref{DirectProductsLocalitiesInclusions}(b). Hence, $\phi$ is an isomorphism of partial groups by Lemma~\ref{IsomorphismOfPartialGroups}(a).   
\end{proof}

From now on we assume that $(\L,\Delta,S)$ is a locality.

\begin{definition}
Let  $(\L_1,\Delta_1,S_1)$ and $(\L_2,\Delta_2,S_2)$ be sublocalities of $\L$. We say that the locality $(\L,\Delta,S)$ is the \textit{(internal) central product of the localities} $(\L_1,\Delta_1,S_1)$ and $(\L_2,\Delta_2,S_2)$ if the following conditions hold:
\begin{itemize}
\item $\L$ is the internal central product of $\L_1$ and $\L_2$ as a partial group, 
\item $S=S_1S_2$, and 
\item $\Delta$ is the set of subgroups of $S$ containing a subgroup of the form $P_1P_2$ with $P_i\in\Delta_i$ for $i=1,2$.
\end{itemize}
If in addition to these properties (D) holds, i.e. if $\L$ is the internal direct product of $\L_1$ and $\L_2$, then we call $(\L,\Delta,S)$ the \textit{(internal) direct product of the localities} $(\L_1,\Delta_1,S_1)$ and $(\L_2,\Delta_2,S_2)$. 
\end{definition}

\begin{example}\label{CentralProductExternalInternal}
Let $(\L_1,\Delta_1,S_1)$ and $(\L_2,\Delta_2,S_2)$ be localities. Write $(\L,\Delta,S)$ for the external direct product of the localities $(\L_1,\Delta_1,S_1)$ and $(\L_2,\Delta_2,S_2)$, i.e. $\L=\L_1\times\L_2$, $S=S_1\times S_2$, and $\Delta=\Delta_1*\Delta_2$ is the set of subgroups of $S$ containing a subgroup of the form $P_1\times P_2$ with $P_i\in\Delta_i$ for $i=1,2$. Set $\hat{\Delta}_i:=\Delta_i\iota_i$ and $\hat{S}_i=S_i\iota_i$ for $i=1,2$. 
\begin{itemize}
\item [(a)] For each $i=1,2$, $(\hat{\L}_i,\hat{\Delta}_i,\hat{S}_i)$ is a sublocality of $(\L,\Delta,S)$. Moreover, $(\L,\Delta,S)$ is the internal direct product of the localities $(\hat{\L}_1,\hat{\Delta}_1,\hat{S}_1)$ and $(\hat{\L}_2,\hat{\Delta}_2,\hat{S}_2)$.
\item [(b)] Let $Z\leq Z(\L)$ with $Z\cap\hat{\L}_i=\{\One\}$ and let $\rho\colon \L\rightarrow\L/Z$ so that $(\L/Z,\Delta\rho,S\rho)$ is the external central product of $(\L_1,\Delta_1,S_1)$ and $(\L_2,\Delta_2,S_2)$ over $Z$. Then for $i=1,2$, $(\hat{\L}_i\rho,\hat{\Delta}_i\rho,\hat{S}_i\rho)$ is a sublocality of $(\L/Z,\Delta\rho,S\rho)$ and $\rho|_{\hat{\L}_i}\colon \hat{\L}_i\rightarrow \hat{\L}_i\rho$ is a projection of localities from $(\hat{\L}_i,\hat{\Delta}_i,\hat{S}_i)$ to $(\hat{\L}_i\rho,\hat{\Delta}_i\rho,\hat{S}_i\rho)$. Moreover, $(\L/Z,\Delta\rho,S\rho)$ is an internal central product of $(\hat{\L}_1\rho,\hat{\Delta}_1\rho,\hat{S}_1\rho)$ and $(\hat{\L}_2\rho,\hat{\Delta}_2\rho,\hat{S}_2\rho)$. 
\end{itemize} 
\end{example}

\begin{proof}
By Lemma~\ref{DirectProductLiSublocality}, $(\hat{\L}_i,\hat{\Delta}_i,\hat{S}_i)$ is a sublocality of $(\L,\Delta,S)$. As seen in Example~\ref{DirectProductPartialGroupExternalInternal}, $\L=\L_1\times\L_2$ is an internal direct product of $\hat{\L}_1$ and $\hat{\L}_2$. It is now immediate that $(\L,\Delta,S)$ is an internal direct product of the sublocalities $(\hat{\L}_1,\hat{\Delta}_1,\hat{S}_1)$ and $(\hat{\L}_2,\hat{\Delta}_2,\hat{S}_2)$. This shows (a). Property (b) follows now from (a) and Lemma~\ref{CentralProductTranslateProjection} below.
\end{proof}

\textbf{Suppose from now on that  $(\L_1,\Delta_1,S_1)$ and $(\L_2,\Delta_2,S_2)$ are sublocalities of $\L$.}

\begin{lemma}\label{CentralProductTranslateProjection}
Suppose $(\L,\Delta,S)$ is the internal central product of $(\L_1,\Delta_1,S_1)$ and $(\L_2,\Delta_2,S_2)$. Let $(\L',\Delta',S')$ be a locality and let $\beta\colon\L\rightarrow\L'$ be a projection of localities from $(\L,\Delta,S)$ to $(\L',\Delta',S')$ with $\ker(\beta)\subseteq Z(\L)$. 

\smallskip

\noindent Then $(\L_i\beta,\Delta_i\beta,S_i\beta)$ is a sublocality of $(\L',\Delta',S')$ for $i=1,2$ and $\beta|_{\L_i}\colon \L_i\rightarrow\L_i\beta$ is a projection of localities from $(\L_i,\Delta_i,S_i)$ to $(\L_i\beta,\Delta_i\beta,S_i\beta)$. Moreover, $(\L',\Delta',S')$ is the internal central product of the sublocalities $(\L_1\beta,\Delta_1\beta,S_1\beta)$ and $(\L_2\beta,\Delta_2\beta,S_2\beta)$.
\end{lemma}

\begin{proof}
By Lemma~\ref{SublocalityUnderProjection}, $(\L_i\beta,\Delta_i\beta,S_i\beta)$ is a sublocality of $(\L',\Delta',S')$ for $i=1,2$ and $\beta|_{\L_i}\colon \L_i\rightarrow\L_i\beta$ is a projection of localities from $(\L_i,\Delta_i,S_i)$ to $(\L_i\beta,\Delta_i\beta,S_i\beta)$. 

\smallskip

We show next that $\L'$ is the internal central product of $\L_1\beta$ and $\L_2\beta$ as a partial group. Write $\Pi'\colon\D'\rightarrow\L'$ for the partial product on $\L'$. Set  
\begin{eqnarray*}
\D^+:=\{(\Pi'(\hat{f}_1,\hat{g}_1),\dots,\Pi'(\hat{f}_n,\hat{g}_n))\colon (\hat{f}_1,\dots,\hat{f}_n)\in \D'\cap\W(\L_1\beta)\;(\hat{g}_1,\dots,\hat{g}_n)\in\D'\cap\W(\L_2\beta),& &\\
(\hat{f}_j,\hat{g}_j)\in\D'\mbox{ for }j=1,\dots,n\}&.&
\end{eqnarray*}
Showing property (C1) for $\L'$ means to show that $\D^+=\D'$. As $\beta$ is a projection and (C1) holds for $\L$, it is straightforward to check that $\D'=\D\beta^*\subseteq \D^+$. Let now $w\in\D^+$ and write $w=(\Pi'(\hat{f}_1,\hat{g}_1),\dots,\Pi'(\hat{f}_n,\hat{g}_n))$ with $w_1:=(\hat{f}_1,\dots,\hat{f}_n)\in \D'\cap\W(\L_1\beta)$, $w_2:=(\hat{g}_1,\dots,\hat{g}_n)\in\D'\cap\W(\L_2\beta)$, and $(\hat{f}_j,\hat{g}_j)\in\D'$ for $j=1,\dots,n$. For $j=1,\dots,n$ let $f_j\in\L_1$ and $g_j\in\L_2$ with $f_j\beta=\hat{f}_j$ and $g_j\beta=\hat{g}_j$. Set $v_1:=(f_1,\dots,f_n)$ and $v_2=(f_1,\dots,f_2)$. Note that $v_i\beta^*=w_i\in\D'$ for $i=1,2$, and $(f_j,g_j)\beta^*=(\hat{f}_j,\hat{g}_j)\in\D'$ for $j=1,\dots,n$. So by Lemma~\ref{ModCentral1}, $v_i\in\D$ for $i=1,2$ and $(f_j,g_j)\in\D$ for $j=1,\dots,n$. Hence, since (C1) holds for $\L$, we have $v:=(\Pi(f_1,g_1),\dots,\Pi(f_n,g_n))\in\D$. As $\beta$ is a homomorphism of partial groups, $v\beta^*=w$ and thus $w\in \D\beta^*=\D'$. This shows $\D^+=\D'$ and (C1) holds for $\L'$. Moreover, using that (C2) holds for $\L$ and that $\beta$ is a homomorphism of partial groups, we obtain  $\Pi'(w)=\Pi'(v\beta^*)=(\Pi(v))\beta=(\Pi(\Pi(v_1),\Pi(v_2)))\beta=\Pi'(\Pi'(v_1\beta^*),\Pi'(v_2\beta^*))=\Pi'(\Pi'(\hat{f}_1,\dots,\hat{f}_n),\Pi'(\hat{g}_1,\dots,\hat{g}_n))$. Hence (C2) holds for $\L'$. So $\L'$ is the central product of $\L_1\beta$ and $\L_2\beta$ as a partial group.

\smallskip

Since $S\beta=S'$, $\Delta\beta=\Delta'$ and $(\L,\Delta,S)$ is the internal central product of $\L_1$ and $\L_2$, it is now easy to observe that the assertion holds.  
\end{proof}

\begin{prop}\label{InternalCentralProductsLocalities}
Let $(\L_1\times\L_2,\Delta_1*\Delta_2,S_1\times S_2)$ be the external direct product of the localities $(\L_1,\Delta_1,S_1)$ and $(\L_2,\Delta_2,S_2)$, i.e. $\Delta_1*\Delta_2$ is the set of subgroups of $S_1\times S_2$ containing a subgroup of the form $P_1\times P_2$ with $P_i\in\Delta_i$ for $i=1,2$. Consider the map
\[\phi\colon \L_1\times\L_2\rightarrow \L,\;(f,g)\mapsto \Pi(f,g).\]
Then the following hold:
\begin{itemize}
 \item [(a)] The map $\phi$ is well defined and a projection of localities from $(\L_1\times\L_2,\Delta_1*\Delta_2,S_1\times S_2)$ to $(\L,\Delta,S)$ if and only if $(\L,\Delta,S)$ is the internal central product of the localities $(\L_1,\Delta_1,S_1)$ and $(\L_2,\Delta_2,S_2)$. 
 \item [(b)] Suppose $\phi$ is well-defined and a projection of localities from $(\L_1\times\L_2,\Delta_1*\Delta_2,S_1\times S_2)$ to $(\L,\Delta,S)$. Then the quotient locality 
\[(\L_1\times\L_1,\Delta_1*\Delta_2,S_1\times S_2)/\ker(\phi)\] 
forms an external central product of the localities $(\L_1,\Delta_1,S_1)$ and $(\L_2,\Delta_2,S_2)$, and $\phi$ induced an isomorphism of localities
\[(\L_1\times\L_2)/\ker(\phi)\rightarrow\L,\;h\ker(\phi)\mapsto h\phi.\]
 \item [(c)] Suppose $\phi$ is well-defined and a projection between the localities $(\L_1\times\L_2,\Delta_1*\Delta_2,S_1\times S_2)$ and $(\L,\Delta,S)$. Then the following are equivalent:
\begin{itemize}
\item [(i)] $\phi$ is an isomorphism of localities,
\item [(ii)] $\ker(\phi)=\{\One\}$,
\item [(iii)] $(\L,\Delta,S)$ is the internal direct product of $(\L_1,\Delta_1,S_1)$ and $(\L_2,\Delta_2,S_2)$,
\item [(iv)] $\L_1\cap\L_2=\{\One\}$. 
\end{itemize}
\end{itemize}
\end{prop}

\begin{proof}
 Suppose $\phi$ is well-defined. Then $(S_1\times S_2)\phi=S_1S_2$ and $(\Delta_1*\Delta_2)\phi=\{Q\phi\colon Q\in\Delta_1*\Delta_2\}$ is the set of subgroups of $S$ containing a subgroup of the form $P_1P_2$ with $P_i\in\Delta_i$ for $i=1,2$. Hence, (a) follows from  Lemma~\ref{InternalCentralProductsPartialGroups}(a). Assume now that $\phi$ is a projection between the localities $(\L_1\times\L_2,\Delta_1*\Delta_2,S_1\times S_2)$ and $(\L,\Delta,S)$. Then $\ker(\phi)$ is a partial normal subgroup and we can form the quotient locality $(\L_1\times\L_2)/\ker(\phi)$ and by \cite[Theorem~4.7]{Chermak:2015}, $\phi$ induces an isomorphism between the localities $(\L_1\times\L_2,\Delta_1*\Delta_2,S_1\times S_2)/\ker(\phi)$ and $(\L,\Delta,S)$. By Lemma~\ref{InternalCentralProductsPartialGroups}(b), $(\L_1\times\L_2,\Delta_1*\Delta_2,S_1\times S_2)/\ker(\phi)$ forms an external central product of $(\L_1,\Delta_1,S_1)$ and $(\L_2,\Delta_2,S_2)$. This proves (b).

\smallskip

By \cite[Theorem~4.4(d)]{Chermak:2015}, a projection between localities is an isomorphism if and only if its kernel is trivial. Hence, properties (i) and (ii) in part (c) are equivalent. Properties (i) and (iii) are equivalent by (a) and Lemma~\ref{InternalCentralProductsPartialGroups}(c). By Lemma~\ref{InternalCentralProductsPartialGroups}(b), $\ker(\phi)=\{(f,f^{-1})\colon f\in\L_1\cap\L_2\}$ which implies that (ii) and (iv) are equivalent. 
\end{proof}

\begin{lemma}
 Suppose $(\L,\Delta,S)$ is the internal central product of $(\L_1,\Delta_1,S_1)$ and $(\L_2,\Delta_2,S_2)$. Then $\L_1$ and $\L_2$ are partial normal subgroups of $\L$.
\end{lemma}

\begin{proof}
 Let $i\in\{1,2\}$. By Proposition~\ref{DirectProductsLocalitiesInclusions}(a), $\hat{\L}_i=\L_i\iota_i$ is a partial normal subgroup of $\L_1\times\L_2$. By Lemma~\ref{InternalCentralProductsLocalities}(a), the map $\phi\colon \L_1\times\L_2\rightarrow \L,\;(f,g)\mapsto \Pi(f,g)$ is well-defined and a projection between the localities $(\L_1\times\L_2,\Delta_1*\Delta_2,S_1\times S_2)$ and $(\L,\Delta,S)$. So by Lemma~\ref{LocalitiesProjectionsPartialNormal}, $\hat{\L_i}\phi$ is a partial normal subgroup of $\L$. By Proposition~\ref{InternalCentralProductsPartialGroups}(d), we have $\hat{\L_i}\phi=\L_i$ and thus the assertion follows. 
\end{proof}

\begin{lemma}\label{InternalCentralProductLinkingLocality}
Suppose $(\L,\Delta,S)$ is the internal central product of $(\L_1,\Delta_1,S_1)$ and $(\L_2,\Delta_2,S_2)$. Then $(\L_i,\Delta_i,S_i)$ is of objective characteristic $p$ for $i=1,2$ if and only if $\L_1\cap \L_2\leq S_1\cap S_2$ and $(\L,\Delta,S)$ is of objective characteristic $p$. Similarly, $(\L_i,\Delta_i,S_i)$ is a linking locality for $i=1,2$ if and only if $\L_1\cap \L_2\leq S_1\cap S_2$ and $(\L,\Delta,S)$ is a linking locality.
\end{lemma}

\begin{proof}
 By Proposition~\ref{InternalCentralProductsLocalities}(a), the map $\phi\colon \L_1\times\L_2\rightarrow \L,\;(f,g)\mapsto \Pi(f,g)$ is well-defined and a projection of localities from $(\L_1\times\L_2,\Delta_1*\Delta_2,S_1\times S_2)$ to $(\L,\Delta,S)$. By Proposition~\ref{InternalCentralProductsPartialGroups}(b), $Z:=\ker(\phi)=\{(f,f^{-1})\colon f\in \L_1\cap\L_2\}\leq Z(\L_1\times\L_2)$. So by \cite[Proposition~9.3]{Henke:2015}, $(\L_1\times\L_2,\Delta_1*\Delta_2,S_1\times S_2)$ is of objective characteristic $p$ if and only if $Z\leq S_1\times S_2$ and $(\L,\Delta,S)$ is of objective characteristic $p$; and $(\L_1\times\L_2,\Delta_1*\Delta_2,S_1\times S_2)$ is a linking locality if and only if $Z\leq S_1\times S_2$ and $(\L,\Delta,S)$ is a linking locality. Note that $Z=\{(f,f^{-1})\colon f\in\L_1\cap\L_2\}\subseteq S_1\times S_2$ if and only if $\L_1\cap \L_2\leq S_1\cap S_2$. So the assertion follows from Lemma~\ref{DirectProductObjectiveCharp} and Lemma~\ref{DirectProductLinkingLocality}. 
\end{proof}

\begin{prop}\label{LastProposition}
 Let $\F$ be a saturated fusion system over $S$ such that $\F$ is the internal central product of two subsystems $\F_1$ and $\F_2$ over $S_1$ and $S_2$ respectively. For $i=1,2$ let $\F_i^{cr}\subseteq\Delta_i\subseteq \F_i^s$ such that $\Delta_i$ is closed under taking $\F_i$-conjugates and overgroups in $S_i$. Let $\Delta$ be the set of overgroups in $S$ of the subgroups of the form $P_1P_2$ with $P_i\in\Delta_i$ for $i=1,2$. 
\begin{itemize}
\item [(a)] The set $\Delta$ is closed under taking $\F$-conjugates and overgroups in $S$. Moreover, $\F^{cr}\subseteq\Delta\subseteq \F^s$. 
\item [(b)] Suppose $(\L,\Delta,S)$ is a linking locality over $\F$. Then $\L$ is the central product of two sublocalities $(\L_1,\Delta_1,S_1)$ and $(\L_2,\Delta_2,S_2)$ such that $\F_i=\F_{S_i}(\L_i)$ for $i=1,2$.
\end{itemize}
\end{prop}

\begin{proof}
By Lemma~\ref{CentralProductFusionSystems}(c), $\Delta$ is closed under taking $\F$-conjugates and overgroups in $S$. By Lemma~\ref{CentralProductFusionSystems}(a), $\F^{cr}\subseteq \Delta$. By \cite[Theorem~A(b)]{Henke:2015}, $\F^s$ is closed under taking overgroups. So Lemma~\ref{CentralProductFusionSystems}(b) implies $\Delta\subseteq\F^s$. This proves (a).

\smallskip

Write $\hat{\F}_i$ for the canonical image of $\F_i$ in $\F_1\times \F_2$.  
As $\F$ is the central product of $\F_1$ and $\F_2$, the map $\alpha\colon S_1\times S_2\rightarrow S,(s_1,s_2)\mapsto s_1s_2$ induces an epimorphism from $\F_1\times \F_2$ to $\F$ with $\hat{\F}_i\alpha=\F_i$. We have seen that  $Z:=\ker(\alpha)\leq Z(\F_1\times \F_2)$. 

\smallskip

For $i=1,2$ let $(\M_i,\Delta_i,S_i)$ be a linking locality over $\F_i$. Set $\M=\M_1\times \M_2$, $\Gamma=\Delta_1*\Delta_2$ and $T=S_1\times S_2$. Then $(\M,\Gamma,T)$ is a locality over $\F_1\times \F_2$ by Lemma~\ref{DirectProductIsLocality}. For $i=1,2$ let $\iota_i\colon \M_i\rightarrow \M_1\times\M_2$ be the inclusion map. By Lemma~\ref{DirectProductLiSublocality}, $(\M_i\iota_i,\Delta_i\iota_i,S_i\iota_i)$ is a sublocality of $(\M,\Gamma,T)$ with $\F_{S_i\iota_i}(\M_i\iota_i)=\hat{\F}_i$. 

\smallskip

As $(\M_i,\Delta_i,S_i)$ is a linking locality for $i=1,2$, $(\M,\Gamma,T)$ is a linking locality by Lemma~\ref{DirectProductLinkingLocality}. Hence, by \cite[Proposition~4]{Henke:2015}, $Z\leq Z(\M)$. Let $\rho\colon \M\rightarrow \M/Z$ be the canonical projection so that $(\M/Z,\Gamma\rho,T\rho)$ is the central product of $(\M_1,\Delta_1,S_1)$ and $(\M_2,\Delta_2,S_2)$ over $Z$. As seen in Example~\ref{CentralProductExternalInternal}, $(\M_i\iota_i\rho,\Delta_i\iota_i\rho,S_i\iota_i\rho)$ is a sublocality of $(\M/Z,\Gamma\rho,T\rho)$, and $\rho|_{\M_i\iota_i}\colon \M_i\iota_i\rightarrow \M_i\iota_i\rho$ is a projection of localities from $(\M_i\iota_i,\Delta_i\iota_i,S_i\iota_i)$ to $(\M_i\iota_i\rho,\Delta_i\iota_i\rho,S_i\iota_i\rho)$ for $i=1,2$. Moreover, $(\M/Z,\Gamma\rho,T\rho)$ is an internal central product of $(\M_1\iota_1\rho,\Delta_1\iota_1\rho,S_1\iota_1\rho)$ and $(\M_2\iota_2\rho,\Delta_2\iota_2\rho,S_2\iota_2\rho)$. 

\smallskip

By Lemma~\ref{ExternalCentralProductLemma}(b),(c), $(\M/Z,\Gamma\rho,T\rho)$ is a linking locality over $(\F_1\times\F_2)/Z$. As observed before, the map $\ov{\alpha}\colon T/Z\rightarrow S,Zt\mapsto t\alpha$ induces an isomorphism from $(\F_1\times\F_2)/Z$ to $\F$. Observe also that $(\rho|_T)\circ\ov{\alpha}=\alpha$ and thus $\Gamma\rho\ov{\alpha}=\Gamma\alpha=\Delta$. Hence, by Proposition~\ref{IsoFusionIsoLinkingLocality}, there exists $\beta\colon \M/Z\rightarrow \L$ such that $\beta$ is an isomorphism of localities from $(\M/Z,\Gamma\rho,T\rho)$ to $(\L,\Delta,S)$. Set $\L_i:=\M_i\iota_i\rho\beta$. As $(\rho|_T)\circ \ov{\alpha}=\alpha$, we have $\iota_i\circ\rho\circ\beta=\iota_i\circ\alpha=\id_{S_i}$ for $i=1,2$. In particular, $S_i\iota_i\rho\beta=S_i$ and $\Delta_i\iota_i\rho\beta=\Delta_i$ for $i=1,2$. So by Lemma~\ref{CentralProductTranslateProjection}, $(\L_i,\Delta_i,S_i)$ is a sublocality of $(\L,\Delta,S)$ for $i=1,2$,  $\beta|_{\M_i\iota_i\rho}\colon \M_i\iota_i\rho\rightarrow\L_i$ is a projection of localities from $(\M_i\iota_i\rho,\Delta_i\iota_i\rho,S_i\iota_i\rho)$ to $(\L_i,\Delta_i,S)$, and $(\L,\Delta,S)$ is a central product of $(\L_1,\Delta_1,S_1)$ and $(\L_2,\Delta_2,S_2)$. Observe that the composition of projections of localities is a projection of localities again. Hence, $\rho\circ\beta$ is a projection of localities from $(\M,\Gamma,T)$ to $(\L,\Delta,S)$, and for $i=1,2$,  $(\rho\circ\beta)|_{\M_i\iota_i}$ is a projection of localities from $(\M_i\iota_i,\Delta\iota_i,S_i\iota_i)$ to $(\L_i,\Delta_i,S_i)$. Hence, by \cite[Theorem~5.7(b)]{Henke:2015}, $(\rho\circ\beta)|_{S_i\iota_i}=\alpha|_{S_i\iota_i}$ induces an epimorphism from $\hat{\F}_i=\F_{S_i\iota_i}(\M_i\iota_i)$ to $\F_{S_1}(\L_1)$. Hence, $\F_{S_1}(\L_1)=\hat{\F}_i(\rho\circ\beta)|_{S_i\iota_i}=\hat{\F}_i\alpha=\F_i$. This completes the proof.
\end{proof}

\bibliographystyle{amsplain}
\bibliography{gpfus}

\end{document}